\documentclass[12pt,leqno,a4paper]{amsart}
\usepackage{amssymb,enumerate}
\newcommand{\FF}{{\mathbb{F}}}

\newcommand{\fA}{{\mathfrak{A}}}
\newcommand{\fS}{{\mathfrak{S}}}

\newcommand{\ad}{{\operatorname{ad}}}
\newcommand{\rank}{{\operatorname{rank}}}
\newcommand{\tr}{{\operatorname{tr}}}
\newcommand{\Aut}{{\operatorname{Aut}}}
\newcommand{\GL}{{\operatorname{GL}}}
\newcommand{\PGL}{{\operatorname{PGL}}}
\newcommand{\PSL}{{\operatorname{L}}}
\newcommand{\GU}{{\operatorname{GU}}}
\newcommand{\SU}{{\operatorname{SU}}}
\newcommand{\PSU}{{\operatorname{U}}}
\newcommand{\SL}{{\operatorname{SL}}}
\newcommand{\Sp}{{\operatorname{Sp}}}
\newcommand{\PSp}{{\operatorname{S}}}
\newcommand{\OO}{{\operatorname{O}}}
\newcommand{\SO}{{\operatorname{SO}}}
\newcommand{\GO}{{\operatorname{GO}}}
\newcommand{\CO}{{\operatorname{CO}}}

\newcommand{\Chevie}{{\sf Chevie}}
\newcommand{\GAP}{{\sf GAP}}

\newcommand{\tw}[1]{{}^#1\!}
\def\pmod#1{~({\rm mod}~#1)}

\let\la=\lambda
\let\sg=\sigma

\newtheorem{thm}{Theorem}[section]
\newtheorem{lem}[thm]{Lemma}
\newtheorem{cor}[thm]{Corollary}
\newtheorem{prop}[thm]{Proposition}

\newtheorem*{thmA}{Theorem}

\theoremstyle{definition}
\newtheorem{exmp}[thm]{Example}

\newtheorem{conj}[thm]{Conjecture}

\theoremstyle{remark}
\newtheorem{rem}[thm]{Remark}

\begin{document}

\title[Variations on Baer--Suzuki]{Variations on the Baer--Suzuki Theorem}

\date{\today}

\author{Robert Guralnick}
\address{Department of Mathematics, University of
  Southern California, Los Angeles, CA 90089-2532, USA}
\makeatletter\email{guralnic@usc.edu}\makeatother
\author{Gunter Malle}
\address{FB Mathematik, TU Kaiserslautern, Postfach 3049,
  67653 Kaisers\-lautern, Germany}
  \makeatletter\email{malle@mathematik.uni-kl.de}\makeatother

\thanks{The first author was partially supported by the NSF
  grants DMS-1001962, DMS-1302886  and the Simons Foundation Fellowship 224965.
  The second author gratefully acknowledges financial support by ERC
  Advanced Grant 291512.}



\dedicatory{Dedicated to Bernd Fischer on the occasion of his 75th birthday}

\begin{abstract}
The Baer--Suzuki theorem says that if $p$ is a prime, $x$ is a $p$-element
in a finite group $G$ and $\langle x, x^g \rangle$ is a $p$-group for all
$g \in G$, then the normal closure of $x$ in $G$ is a $p$-group.
We consider the case where $x^g$ is replaced by $y^g$ for some other
$p$-element~$y$. While the analog of Baer--Suzuki is not true, we show that
some variation is. We also answer a closely related question of Pavel
Shumyatsky on commutators of conjugacy classes of $p$-elements.
\end{abstract}

\maketitle

\pagestyle{myheadings}

\section{Introduction}  \label{sec:intro}

The Baer--Suzuki theorem asserts:

\begin{thmA}  
 Let $G$ be a finite group and $x \in G$. If $\langle x, x^g \rangle$
 is nilpotent for all $g \in G$, then $\langle x^G \rangle$ is a nilpotent
 normal subgroup of $G$.
\end{thmA}

There are many relatively elementary proofs of this (see \cite{AL},
\cite[p.~298]{HB} or \cite[p.~196]{Suz}). Clearly, it suffices to prove
the result for $x$ a $p$-element for each prime $p$ (or indeed
for $x$ of prime order $p$).  We were recently informed by Bernd Fischer
that Reinhold Baer had asked what one can say if instead for given $x, y\in G$
we have the hypothesis that $\langle x, y^g \rangle$ is a $p$-group for all
$g\in G$.  Examples in \cite{GMT} show that there is not too much to say
in general. The most one could expect is that $[x^G,y^G] \le O_p(G)$ and
this can fail.

Using the classification of finite simple groups, the following generalization
of Baer--Suzuki for primes at least $5$ was proved in \cite[Thm.~1.9]{GMT}
though:

\begin{thmA}   
 Let $G$ be a finite group and $p \ge 5$ a prime. If $C$ and $D$ are
 normal subsets of $G$ with $\langle C \rangle = \langle D \rangle$
 and $\langle c, d \rangle$ is a $p$-group for all $(c,d) \in C \times D$,
 then $\langle C \rangle$ is a normal $p$-subgroup of $G$.
\end{thmA}

The conclusion fails for $p=2,3$ and fails without the assumption that
$\langle C \rangle = \langle D \rangle$.
In this note, we show that if we strengthen the assumption on the structure
of $\langle c, d \rangle$, then we can prove a version of Baer--Suzuki:

\begin{thm}   \label{thm:main}
 Let $G$ be a finite group and $p$ a prime. Let $C, D$ be normal subsets of
 $G$ such that if $(c,d) \in C \times D$, then $\langle c, d \rangle$ is a
 $p$-group with no section isomorphic to $Z_p \wr Z_p$. Then $[C,D] \le O_p(G)$.
\end{thm}

We also classify in Theorems~\ref{thm:p=3} and~\ref{thm:p=2} all pairs of
conjugacy classes $C,D$ of elements of order $p$ in finite almost simple
groups such that $\langle c,d \rangle$ is a $p$-group for all
$(c,d)\in C \times D$ (by \cite[Thm.~8.4]{GMT} this can only happen for
$p\le 3$). We call such a pair of classes a \emph{Baer--Fischer pair}
in view of Baer's question and of the fact that such pairs for $p=2$ were
found by Fischer in the automorphism group of his smallest group $Fi_{22}$ and
in the involution centralizer $2.\tw2E_6(2).2$ of the baby monster $B$.

Remark that for the case $p=2$ our result is somewhat complementary to various
earlier investigations by Fischer, Aschbacher, Timmesfeld, and others on
groups generated by 3-transpositions, or by odd involutions, which considered
involution classes for which the products mostly have odd order, instead
of 2-power order as here (see for example the survey \cite{Ti}).

The second goal of this paper is to answer a question of Pavel Shumyatsky.

\begin{thm}   \label{thm:mainB}
 Let $p$ be prime. Let $G$ be a finite group with a normal subset $C$
 consisting of $p$-elements and closed under taking commutators. Then
 one of the following occurs:
 \begin{enumerate}
  \item[\rm(1)] $\langle C \rangle$ is a $p$-group; or
  \item[\rm(2)] $p=5$ and $\langle C \rangle O_5(G)/O_5(G)$ is a direct
   product of copies of $\fA_5$ and $C$ is not closed under squares.
 \end{enumerate}
\end{thm}

Note that this result is closely related to the Baer--Suzuki theorem (which
can be viewed as
saying that if $C$ is a normal set of $p$-elements and $\langle x, y\rangle$
is a $p$-group for all $x, y \in C$, then $\langle C \rangle$ is a $p$-group).
Note that if $C$ is a conjugacy class of elements of order $5$ in $\fA_5$,
then $[C,C] = C \cup \{1 \}$, but $\langle C\rangle=\fA_5$ is simple.

\medskip

We conjecture that an even stronger property holds:

\begin{conj}
 Let $5 \ne p$ be a prime. Let $C$ be a conjugacy class of $p$-elements in
 the finite group $G$. If $[c,d]$ is a $p$-element for all $c,d \in G$,
 then $C \subset O_p(G)$.
\end{conj}

Our methods would also provide another proof of a related result
of \cite{GR, X}:

\begin{thm}
 Let $p$ be a prime and $C$ a conjugacy class of $p$-elements in the finite
 group $G$. If $CC^{-1}$ consists of $p$-elements, then $C \subset O_p(G)$.
\end{thm}

\medskip

Analogous situations to Theorems~\ref{thm:main} and~\ref{thm:mainB} for
almost simple algebraic groups were completely classified in our
paper \cite{GM}.

The paper is organized as follows. In the next section, we prove some results
about representations of $p$-groups. In Section~\ref{sec:proof}, we prove
Theorem~\ref{thm:main} assuming the results of Section~\ref{sec:p=3} (for the
prime $3$).
We then classify pairs of conjugacy classes of $3$-elements in almost simple
groups such that every pair generates a 3-group. In Section~\ref{sec:p=2}, we
consider pairs of involutions in almost simple groups.
In Section~\ref{sec:commut}, we handle the case of Theorem~\ref{thm:mainB} with
$p\ne 5$ or when $C$ is closed under squaring. In the final section, we
consider the case when $p=5$. In both
cases, we can reduce to the case of simple groups.
\smallskip

We thank Michael Aschbacher for suggesting some variation on the hypotheses of
Theorem~\ref{thm:main}.  We also thank Pavel Shumyatsky for communicating
his question to us and Thomas Breuer, Klaus Lux, Kay Magaard, and Akos Seress
for very helpful comments.

\section{On $p$-group representations}  \label{sec:repn}

Let $p$ be an odd prime and $k$ be a field of characteristic $r \ne p$.
We write $Z_p$ for the cyclic group of order~$p$.

\begin{lem}   \label{lem:dim}
 Let $P$ be a finite $p$-group and let $V$ be an irreducible $kP$-module
 such that $[P,P]$ is not trivial on $V$. Let $\{x_1, \ldots, x_s\}$ be a
 generating set for $P$.
 \begin{enumerate}[\rm(a)]
  \item There exists $i$ so that $\dim C_V(x_i) \le (1/p) \dim V$.
  \item If each $x_j$ has order $p$, there exists $i$ with
   $\dim C_V(x_i) =(1/p) \dim V$.
 \end{enumerate}
\end{lem}

\begin{proof}
By viewing $V$ as a module over $\mathrm{End}_{kP}(V)$, we may assume that
$V$ is absolutely irreducible and $k$ is algebraically closed. Thus,
$V = \lambda_H^P$ for some $1$-dimensional representation $\la$ of a proper
subgroup $H$ of $P$. Let $M$ be a maximal subgroup of $P$ containing $H$.
Thus, $V = W_M^P = W_1 \oplus \cdots \oplus W_p$ where the $W_i$ are
irreducible $kM$-submodules and $P$ permutes the $W_j$. Some $x = x_i$ must
permute the $W_j$, whence $\dim C_V(x) = \dim C_W(x^p) \le \dim W$.
This proves both (a) and (b).
\end{proof}

\begin{cor}   \label{cor:pgroup}
 In the situation of Lemma~\ref{lem:dim} assume that
 $V = \bigoplus_i [x_i,V]$ and that each $x_i$ has order $p$.
 Then $P$ has a section isomorphic to $Z_p \wr Z_p$.
\end{cor}

\begin{proof}
As in the previous proof, we may write $V = W_1 \oplus \cdots \oplus W_p$
such that the stabilizer of each $W_j$ is a maximal (normal) subgroup $M$ and
$x=x_1$ permutes the non-isomorphic irreducible $kM$-modules $W_i$.
It follows that $x_j$, $j > 1$, is in $M$ because otherwise
$\dim C_V(x_j) = (1/p) \dim V$, contradicting the fact that
$\dim V = \sum \dim[x_i,V]$. Thus, $M$ is the normal closure of
$\langle x_2, \ldots, x_s \rangle$. We can identify each $W_i$ with $W_1$
(as vector spaces) and assume that $x$ just permutes the coordinates.
Writing $x_j=(y_{j1}, \ldots, y_{jp})$ for $j=2, \ldots, s$, we see that
the action of $M$ on $W_1$ is generated by
$\{ y_{ji}\mid j \ge 2, 1 \le i \le s\}$ and that
$\dim W_1 = \sum \dim [y_{ij}, W_1]$.  Since $W_i$ is irreducible for $M$,
this implies that $W_1 = \bigoplus_{i,j}[y_{ij}, W_1]$. If $\dim W_1 > 1$,
it follows by induction that $M$ has a section isomorphic to $Z_p \wr Z_p$.
If $\dim W_1 =1$, then the hypotheses imply that (after reordering), $x_2$ is
a pseudoreflection and $x_j$, $j > 2$, is trivial on $V$.
So $\langle x_1, x_2 \rangle$ acts as $Z_p \wr Z_p$ on $V$, whence the result.
\end{proof}

\section{Proof of Theorem~\ref{thm:main}}  \label{sec:proof}

For $S$ a subset of a finite group $G$, let $e(S)$ denote the largest order
of an element of~$S$.

\begin{thm}
 Let $G$ be a finite group and $p$ a prime. Assume that $C$ and $D$ are
 normal subsets of $G$ such that for all $(c,d) \in C \times D$:
 \begin{enumerate}
  \item[\rm(1)] $\langle c, d \rangle$ is a $p$-group; and
  \item[\rm(2)] no section of $\langle c,d \rangle$ is isomorphic to
   $Z_p\wr Z_p$.
 \end{enumerate}
 Then $[C,D] \le O_p(G)$.
\end{thm}

\begin{proof}
Let $G, C, D$ be a minimal counterexample (say with
$|G| + |C| + |D| + e(C) + e(D)$ minimal). Since the properties are inherited
under homomorphic images, $O_p(G) = 1$. Clearly, $G = \langle C, D \rangle$.

If $C'$ is a proper normal subset of $C$, then by minimality $[C',D]=1$ and
$[C \setminus C', D]=1$. Thus, we may assume that $C$
and $D$ are conjugacy classes of $G$ and that $G=\langle C, D \rangle$.

Let $N$ be a minimal normal subgroup of $G$.  If $M$ is another minimal
normal subgroup, then by minimality, $[C,D]$ is a $p$-group in $G/M$ and
also in $G/N$.  Since $G$ embeds in $G/M \times G/N$, this would imply that
$[C,D]$ is a $p$-group. So $N$ is the unique minimal normal subgroup of $G$.
Note that this implies that either $C_G(N)=1$ or $N$ is abelian
(since $N$ is contained in any nontrivial normal subgroup).

By minimality, $[C,D] = NQ$ where $Q$ is a $p$-group.
If $Q$ centralizes $N$, then $N$ is abelian and so has order prime to $p$,
whence $Q=O_p([C,D]) \le O_p(G)=1$ and $[C,D]=N$.

If the elements of $C$ have order greater than $p$, then by minimality,
$[C^p,D] \le O_p(G)=1$, whence $N \le \langle C^p \rangle$ and so $D$
centralizes $N$.  Thus, $Q \le C_G(N)$ and so $Q=1$ and $[C,D] \le N$ by
the above. Since $N$ is a $p'$-group and $D$ consists of $p$-elements, this
implies that $[C,D]=1$, contradicting the fact that $G$ is a counterexample.
Thus, $C$ and $D$ consist of elements of order $p$. In particular, the result
follows if $p =2 $ (since the hypotheses imply that if $(c,d) \in C \times D$
are involutions, then $cd=dc$).

First suppose that $N$ is an elementary abelian $r$-group for a prime $r$
(necessarily $r \ne p$).  Note that neither $C$ nor $D$ centralizes $N$
(for then $[C,D]=[C^r,D^r]=1$). Choose $(c,d) \in C \times D$ with $[c,d]$
nontrivial. Since $[c,d]$ does not centralize $N$, $[c,d]$ is not in $O_p(H)$
where $H= N\langle c, d \rangle$, whence by minimality $G=H$. In particular,
$G=NQ$ where $Q=\langle c, d \rangle$ is a $p$-group.
Let $H_1 = \langle C \cap Q \rangle$ and $H_2 = \langle D \cap Q \rangle$.

Note that if $(c,d) \in C \times D$, then for any $x \in N$,
$\langle c,d^x \rangle$ is conjugate to a subgroup of $Q$, whence
$N=C_N(c)C_N(d)$. In particular, $\dim [c,N] + \dim [d,N] = \dim N$,
whence $N=[c,N] \oplus [d,N]$ and so by Corollary~\ref{cor:pgroup},
$Q$ has a section isomorphic to $Z_p \wr Z_p$, a contradiction.

Thus, we may assume that $N=L_1 \times\cdots\times L_t$, where $L_i \cong L$
is a nonabelian simple group. Since this is the unique minimal normal subgroup
of $G$, $N=F^*(G)$ has trivial centralizer in $G$. Arguing as above, we may
choose $(c,d) \in C \times D$ with $[c,d] \ne 1$ and $G=NQ$ where
$Q=\langle c, d \rangle$ is a $p$-group.

First suppose that $p$ does not divide $|N|$. Then $Q$ is a Sylow $p$-subgroup
of $NQ$. By Sylow's theorems, $Q$ normalizes a Sylow $\ell$-subgroup of $N$ for
each prime $\ell$. Thus $Q$ normalizes a Sylow $\ell$-subgroup $S$ for some
prime~$\ell$ with $[c,d]$ not centralizing $S$.  By minimality,
$[c,d] \in O_p(SQ)$, a contradiction.

So we may assume that $Q$ normalizes some nontrivial Sylow $p$-subgroup $P$
of $N$.

If $t=1$, the result follows by \cite[Thm.~8.4]{GMT} (for $p>3$) and by
Corollary~\ref{cor:p=3} (for $p=3$).  So assume that $t > 1$.

Suppose that $c$ normalizes each $L_i$.
Then $d$ acts transitively whence $s=p$. Write $c=(c_1, \ldots, c_p)$ where
$c_j \in \Aut(L_j)$ and $d$ permutes the coordinates. Assume that $c_1 \ne 1$.
If $c_j=1$ for all $j > 1$, then $\langle c, d \rangle \cong Z_p \wr Z_p$,
a contradiction.

So suppose that some $c_j \ne 1$ for $j > 1$. Let $y=c^x$ where $x \in L_j$.
Then $\langle y, d \rangle$ is a $p$-group, whence the group generated by
$\langle c_1, c_j^x \rangle$ is a $p$-group in $\Aut(L_1)$ for any
$x \in L_1$. Now apply Corollary~\ref{cor:p=3} (see also \cite[Thm.~8.4]{GMT}
for $p\ge5$).

So we may assume that $c$ and $d$ both induce nontrivial permutations
on $\{L_1, \ldots, L_s\}$. Let $J=N_G(P)$. By \cite[Thm.~X.8.13]{HB} and
\cite[Thm.~1.1]{GMN}, $(J \cap N)/P$ is nontrivial.  By minimality,
$[c,d] \in O_p(J/P)$, whence $[c,d]$ centralizes $(J \cap N)/P$.
In particular, $[c,d]$ must normalize each $L_i$. Thus, $\langle c, d \rangle$
induces an abelian group of permutations on the set of components. Thus,
$s = p$ or $p^2$.  Since $\langle c, d^x \rangle$ is a $p$-group for any
$x \in J \cap N$, we see that $C_{J\cap N/P}(c)C_{J\cap N/P}(d) = J\cap N/P$,
but since $c$ and $d$ act semiregularly on the set of components,
this is not the case.
\end{proof}

\section{Pairs of conjugacy classes of $3$-elements}   \label{sec:p=3}
In this section we classify Baer--Fischer pairs of 3-elements in finite
non-abelian almost simple groups.

It turns out that all examples are finite analogues of examples for possibly
disconnected almost simple algebraic groups, as classified in \cite{GMT}
and \cite{GM}. More precisely they can be obtained as follows:

\begin{exmp}  \label{exmp:D4b}
Let $k$ be an algebraically closed field of characteristic~3 and
$G=\SO_8(k).3$, the extension of the simple algebraic group $G^\circ=\SO_8(k)$
by a graph automorphism of order~3. Let $C_1$ be the class of root elements
in $G^\circ$, with centralizer of type $3A_1$, and $C_2$ the class of the
graph automorphism with centralizer $G_2(k)$ in $G^\circ$. In
\cite[Ex.~3.4]{GM} we showed that $C_1C_2$ consists of 3-elements.
\par
In the finite group $G_0=\OO^+_8(3).3$ an explicit computation with the
character table yields that $(C_1\cap G_0)(C_2\cap G_0)$ consists of three
conjugacy classes with representatives the elements denoted $u_2,u_3,u_4$ in
\cite[Tab.~8]{MDec}, of orders 3,9,9 respectively. More precisely the elements
of $D_1\cap G_0$ are hit once, those from $D_2\cap G_0$ thrice, and those
from one of the three rational classes in $D_3\cap G_0$ are hit six times
by $C_1C_2$, where for $1\le i\le 3$ we let $D_i$ denote the class of $u_{i+1}$
in $G$. A calculation with the centralizer orders then shows that the same is
true for all groups $\OO_8^+(3^a).3$. As both classes $C_1$ and $C_2$ have
non-empty
intersection with $\tw3D_4(3^a).3\le\OO_8^+(3^a).3$, this also gives a pair of
classes of 3-elements in $\tw3D_4(3^a).3$ with all products being 3-elements.
\par
Explicit computation in $\OO_8^+(3).3$ shows that there are pairs
$(x,y)\in C_1\times C_2$ such that $xy\in D_3$ and $\langle x,y\rangle$ has
order~243. Since $D_3$ is the class of maximal dimension among the $D_i$, it
is dense in $C_1C_2$, and so all pairs in $C_1\times C_2$
generate a 3-group. Thus we get examples in $\OO_8^+(3^a).3$. Now both
classes $C_1,C_2$ are stabilized by the graph-field automorphisms of
$\SO_8(k).3$, whence we also obtain such examples for $\tw3D_4(3^a)$. The
remark at the end of \cite[Ex.~3.4]{GM} shows that we generate the same 3-group
when taking suitable long and short root elements in $G_2(3)$.
\end{exmp}

We adopt the notation for outer automorphisms of Lie type groups from
\cite[2.5.13]{GLS}. Thus, in particular graph-field automorphisms only exist
for untwisted groups, and for twisted groups, field automorphisms have order
prime to the order of the twisting.

\begin{thm}   \label{thm:p=3}
 Let $G$ be a finite almost simple group. Suppose that $c, d \in G$ are
 non-trivial $3$-elements such that $\langle c, d^g \rangle$ is a $3$-group
 for all $g \in G$.  Then one of the following holds (up to interchanging
 $c$ and $d$):
 \begin{enumerate}[\rm(1)]
  \item $G=G_2(3^a)$, $c$ is a long root element and $d$ is a short root
   element;
  \item $F^*(G)= \OO^+_8(3^a)$, $c$ is an inner 3-central element and $d$ a
   graph automorphism of order~3 with centralizer $G_2(3^a)$; or
  \item $F^*(G)= \tw3D_4(3^a)$, $c$ is an inner 3-central element and $d$ a
   graph automorphism of order~3 with centralizer $G_2(3^a)$.
 \end{enumerate}
\end{thm}

\begin{proof}
Let $S=F^*(G)$. We consider the various possibilities for $S$ according to the
classification of finite simple groups.

\smallskip
Case 1. $S$ is not of Lie type.

For $S$ sporadic, a calculation of structure constants using the known
character tables shows that no example arises. For $S=\fA_n$, $n\ge5$, there
are no cases by \cite[Lemma~8.2]{GMT}.

\smallskip
Case 2. $S$ of Lie type in characteristic~$p=3$.

If both $c,d$ are contained in $S$, then by \cite[Thm.~4.6]{GMT} the only
examples are those in~(1) of the conclusion.
Now assume that $d$ induces a field or graph-field automorphism on $S$. If $S$
has rank~1, then $S=S(q)\in\{\PSL_2(q),\PSU_3(q),{}^2G_2(q^2)\}$.
By \cite[Prop.~4.9.1]{GLS} there is a unique class of cyclic subgroups of
such automorphisms of order~$3$, and every unipotent element of $S$ is
conjugate to one in $C_S(d)=S(q_0)$, where $q=q_0^3$. Again by
\cite[Prop.~4.9.1]{GLS}, $d$ is conjugate to a non-central element of
$S(q_0)\times\langle d\rangle$, so we reduce to the simple group $S(q_0)$ for
which we are done by induction, except when $S=\PSL_2(3^3)$ or $^2G_2(3^3)$.
In the latter cases, explicit computation shows that there are no examples.
\par
If $S$ has rank at least~2, let's exclude for the moment the case that $S$ is
of type $D_4$ and $c$ or $d$ induce a graph or graph-field automorphism. We
let $P$ be an end node parabolic subgroup with $d$ not contained in its
unipotent radical $Q$.
Then $N_G(P)$ contains a Sylow 3-subgroup of $G$, so we may assume that
$c,d\in N_G(P)$. Now $P/Q$ has a unique non-abelian simple section, on which
both $c,d$ act nontrivially. In this case we are done by induction unless
$P/Q$ is as in~(1), (2) or~(3) of our conclusion. Clearly (1) and (3) cannot
occur as proper Levi factors, and~(2) does not arise since $c,d$ can not induce
graph automorphisms of 3-power order on the Levi factor.
\par
A graph-field automorphism $d$ of $S=\OO_8^+(3^{3a})$ of order~3 normalizes
a subgroup $M=\OO_8^+(3^a)$ which contains representatives for all classes
of elements of order~3 in $S$, and on which it acts by a graph automorphism.
Since all subgroups of order~3 in $S.\langle d\rangle$ are conjugate under
$\Aut(S)$ by \cite[Prop.~4.9.1(e)]{GLS}, a conjugate of $d$ acts as the graph
automorphism of $M$ with parabolic centralizer, whence we do not get examples
by the previously discussed case.
\par
Next assume that $d$ induces a graph automorphism on $S\cong \OO_8^+(3^a)$.
By \cite[Tab.~8]{MDec} there are two such outer automorphisms of order~3 up
to conjugation and inversion, one with centralizer $G_2(q)$ and the other
with centralizer contained inside a parabolic subgroup of $G_2(q)$.
Explicit computation of structure constants shows that the only case for
$\OO_8^+(3)$ is with $c$ an inner 3-central element and $d$ inducing
a graph automorphism with centralizer $G_2(3)$. By Example~\ref{exmp:D4b}
this gives rise to the family of examples in~(2) for $\OO_8^+(3^a)$.
\par
Finally, let $S=\tw3D_4(3^a)$. Here outer automorphisms of order three
stabilize and act non-trivially on a parabolic subgroup $P$ with Levi subgroup
of type $A_1(q^3)$. All non-trivial unipotent classes of $S$ except for the
one of long root elements have representatives outside the unipotent radical
of $P$ (see \cite[Tab.~A.8, A.10]{Him}) and we are done by
induction. Now assume that $c$ is a long root element. Again by
\cite[Tab.~8]{MDec} and \cite[Prop.~5]{MGreen}, there are two classes of
outer automorphisms of order~3 up to inversion. The class whose elements
have centralizer of type $G_2$ leads to case~(3) by Example~\ref{exmp:D4b}.
The other class contains the product of the graph automorphism with a long
root element in its centralizer $G_2(q)$. But the product of two long root
elements in $G_2(q)$ can have even order, whence we do not get an example.
\par

\smallskip
Case 3. $S$ of Lie type in characteristic $p\ne3$.

Here, both $c,d$ are semisimple. In this case, we imitate the argument in the
proof of \cite[Thm.~8.4]{GMT} for the case $p\ge5$, and just comment on the
differences. First assume that $c,d$ both have order~3. If $c$ is inner and
$d$ induces a field or graph-field automorphism, then we may invoke
\cite[Lemma~8.6]{GMT} to descend to a group over a subfield, unless
$S=\tw3D_4(q)$.
We then continue as in \cite[8.2]{GMT} and see that for classical groups we
only need to worry about the case when $S=\OO_8^+(q)$ and $d$, say, is a graph
automorphism. Now $d$ normalizes a subgroup $\OO_8^+(2)$, which contains
representatives from all classes of inner elements of order~3 of $S$.
Computation of structure constants in $\OO_8^+(2).3$ shows that no example
arises. This completes the investigation of classical type groups.
\par
We next discuss exceptional type groups. For $S=\tw2B_2(q^2)$, the only
3-elements are field automorphisms. For $G=G_2(q)$ it can be checked from the
character tables in \cite{Chv} that not both $c,d$ can be inner, and then by
the above cited \cite[Lemma~8.6]{GMT} we reduce to a group over a subfield.
For $\tw3D_4(q)$ all classes of elements of order~3 have representatives in
the subgroup $G_2(q)$.
In $\tw2F_4(q^2)$ there is just one class of elements of order~3, so no example
can exist by the Baer--Suzuki theorem. For the groups of large rank, we use
induction by invoking Lemma~\ref{lem:sub3}.

\smallskip
Case 4. 3-elements of order larger than~$3$.

Clearly we only need to consider elements $c,d$ such that elements of
order~3 in $\langle c\rangle,\langle d\rangle$ are as in~(1)--(3). But all
possibilities for those cases have already been discussed.
\end{proof}

\begin{lem}   \label{lem:sub3}
 Let $G$ be an exceptional group of adjoint Lie type of rank at least~4 in
 characteristic prime to~3. Then all conjugacy classes of elements of order~3
 have (non-central) representatives in a natural subgroup $H$ as listed in
 Table~\ref{tab:exc3}, where $T$ denotes a 1-dimensional split torus.
\end{lem}

\begin{table}[htbp]
 \caption{Subgroups intersecting all classes of elements of order~3}
  \label{tab:exc3}
\[\begin{array}{|l||l|l|}
\hline
     G& H& \text{conditions} \cr
\hline
         F_4(q)& B_4(q)& \cr
     E_6(q)_\ad& A_5(q)T& q\equiv1\pmod3 \cr
               & F_4(q)& q\equiv2\pmod3 \cr
 \tw2E_6(q)_\ad& \tw2A_5(q)T& q\equiv2\pmod3 \cr
               & F_4(q)& q\equiv1\pmod3 \cr
     E_7(q)_\ad& D_6(q)T&  \cr
         E_8(q)& D_8(q)& \cr
\hline
\end{array}\]
\end{table}

\begin{proof}
First assume that $G=F_4(q)$.
The conjugacy classes of elements of order~3 in $G$ and their centralizers are
easily determined using \Chevie\ \cite{Chv}. From this it ensues that
a maximal torus of order $\Phi_1^4$ contains representatives
from all three classes of elements of order~3 when $q\equiv1\pmod3$, which in
turn is contained in a subgroup of type $B_4$, while for $q\equiv2\pmod3$,
the same holds for a maximal torus of order $\Phi_2^4$. In $E_6(q)_\ad$, for
$q\equiv1\pmod3$, all but two classes of elements of order~3 have
representatives in a maximal torus of order~$\Phi_1^6$, which lies inside a
Levi subgroup $A_5(q)T$. For the remaining two
classes, the centralizers $A_2(q^3).3$ and $\tw3D_4(q)\Phi_3$ contain maximal
tori of order $q^6-1$ respectively $(q^3-1)^2$, which also have conjugates in
$A_5(q)T$. The arguments for the remaining cases are completely similar.
\end{proof}

Now we can state the result that we need for the proof of our main
Theorem~\ref{thm:main}.

\begin{cor}   \label{cor:p=3}
 Let $G$ be a finite almost simple group with socle $F^*(G)=S$.  Let $p$ be
 an odd prime.  Let $x, y$ be elements of order $p$ in $G$.
 Then there exists $s \in S$ such that one of the following holds:
 \begin{enumerate}[\rm(1)]
  \item $\langle x, y^s \rangle$ is not a $p$-group; or
  \item  $p=3$ and $Z_3 \wr Z_3$ is a section of $\langle x, y^s \rangle$.
 \end{enumerate}
\end{cor}

\begin{proof}
We may suppose that $G=\langle S, x, y \rangle$. First assume that $p > 3$.
The result follows by \cite[Thm.~8.4]{GMT} except that there $s$ is taken in
$G$ rather than in $S$. If the Sylow $p$-subgroup of $G/S$ is cyclic, then
since $G=SC_G(x)$ or $G=SC_G(y)$, we can take $s \in S$. By the classification,
the only other possibilities are that $S=\PSL_n(q^p)$
or $S=\PSU_n(q^p)$ where $p$ divides $(n, q-1)$ or $(n, q+1)$, respectively.
Note that $G=C_G(x)SC_G(y)$ (and so the result follows) unless $x$ and $y$
are both field automorphisms of order $p$. In this case, after conjugation,
$x$ and $y$ both normalize and do not centralize a subgroup $H$
isomorphic to $\PSL_2(q^p)\cong\PSU_2(q^p)$ and each induce field
automorphisms. By \cite[Prop.~4.9.1]{GLS}, $x$ and $y$ are conjugate in
$\Aut(H)$, whence the result follows by the Baer--Suzuki theorem.

Now assume that $p=3$. Exclude the cases $S=G_2(3^a)$, $S=\OO_8^+(3^a)$
and $S=\tw3D_4(3^a)$ for the moment. Then arguing exactly as for $p > 3$
and using Theorem~\ref{thm:p=3} in place of \cite[Thm.~8.4]{GMT}, we see that
$\langle x, y^s \rangle$ is not a $3$-group for some $s \in S$.

If $S=G_2(3^a)$, then it follows by the earlier results of this section
that either $\langle x, y^s \rangle$ is not a $3$-group for some
$s \in S$ or (up to order), $x$ is a long root element and $y$ is a short
root element. In that case, we see in $G_2(3)$ that $x,y^s$ generate a
subgroup of order $3^5$ (of index~3 in a Sylow $3$-subgroup of $G_2(3)$)
when $xy^s$ has centralizer order $3^3$, and one checks that $Z_3 \wr Z_3$ is
a quotient of that subgroup of $G_2(3)$.

If $S=\OO_8^+(3^a)$, then Theorem~\ref{thm:p=3} shows that
either $\langle x, y^s \rangle$ is not a $3$-group for some $s$ or
(up to order), $x$ is a graph automorphism and $y$ is a $3$-central
element of $S$. Again, explicit computation shows that two conjugates
in $\OO_8^+(3^a).3$ can generate a subgroup of order $3^5$ (see
Example~\ref{exmp:D4b}) which has $Z_3 \wr Z_3$ as a quotient.
This shows the claim also for $S=\tw3D_4(3^a)$.
\end{proof}

\section{Pairs of conjugacy classes of involutions}   \label{sec:p=2}

In this section we classify Baer--Fischer pairs of involution classes in finite
non-abelian almost simple groups. Note that two involutions generate a 2-group
if and only if their product has 2-power order. Before proving the
classification of such pairs, we first give some examples.

\subsection{Baer--Fischer pairs coming from algebraic groups}
Several families of Baer--Fischer pairs are obtained by Galois descent from
corresponding configurations in almost simple algebraic groups.

The Baer--Fischer pairs consisting of unipotent classes in connected groups of
Lie type in characteristic $2$ were classified in \cite[Thm.~4.6]{GMT}.
We next discuss further examples in characteristic~2 coming from configurations
in disconnected algebraic groups as studied in \cite{GM}.

\begin{exmp}   \label{ex:graph-transb}
We continue \cite[Ex.~3.2]{GM} with $C_1$ the conjugacy class of
transvections of $\SL_{2n}(k)$, $n\ge2$, where $k$ is
algebraically closed of characteristic~2, and $C_2$ the class of graph
automorphisms with centralizer $\Sp_{2n}(k)$. Both classes are stable under
the standard Frobenius endomorphism, as well as under unitary Steinberg
endomorphisms of $\SL_{2n}(k)$. Thus we obtain Baer--Fischer pairs of
involution classes both in $\SL_{2n}(q).2$ and $\SU_{2n}(q).2$, where $n\ge2$
and $q=2^a$.
\end{exmp}

\begin{exmp}   \label{ex:ortho p=2b}
We continue \cite[Ex.~3.3]{GM} for the general orthogonal group $G=\GO_{2n}(k)$,
with $n\ge3$ and $k$ algebraically closed of characteristic~2. Let $V$ denote
the natural module for $G$ with invariant symmetric from $(\cdot,\cdot)$. Let
$C_1$ be the class of an involution $x$ with $(xv,v)=0$ for all $v\in V$,
and $C_2$ a class of transvections in $G$. Taking fixed points
under suitable Steinberg endomorphisms we obtain Baer--Fischer pairs for
$\GO_{2n}^{\pm}(2^a)$.
\end{exmp}

\begin{exmp}   \label{exmp:E6b}
We continue \cite[Ex.~3.5]{GM} with $G=E_6(k).2$ the extension of a simple
group $G^\circ=E_6(k)$ of simply connected type $E_6$ by a graph automorphism
of order~2, over
an algebraically closed field $k$ of characteristic~2. Let $C_1$ be the class
of long root elements in $G^\circ$, with centralizer of type $A_5$, and $C_2$
the class of the graph automorphism $\sg$ with centralizer $F_4(k)$ in
$G^\circ$. Here, $C_1C_2$ only contains unipotent elements. \par
Let $D_i$, $i=1,2$, denote the class of the outer unipotent element $u_i$ in
the notation of \cite[Tab.~10]{MGreen}, of order~2 and ~4 respectively.
Representatives of $C_1,C_2$ are also contained in the finite group
$\tw2E_6(2).2$ (noting that by \cite[Prop.~5]{MGreen} the outer unipotent
classes of $E_6(2^a).2$ and $\tw2E_6(2^a).2$ are parametrized in precisely
the same way). An explicit computation of structure constants for the finite
subgroup $G_0=\tw2E_6(2).2$ shows that $C_1C_2\cap G_0$ hits every element of
$D_1\cap G_0$ once, and every element of $D_2\cap G_0$ twice, and no others.
We thus get Baer--Fischer pairs for all the groups $E_6(2^a).2$ and
$\tw2E_6(2^a).2$.
\par
The fact that $G=\tw2E_6(2).2$ is an example can also be seen as follows:
The 2-fold covering of $G$ embeds into the Baby monster $B$ such that the two
above-mentioned involution classes fuse into the class of
$\{3,4\}$-transpositions. The claim follows, as clearly the product of an
inner with an outer element of $G$ has even order. \par
Both classes intersect the maximal subgroup $Fi_{22}.2$ non-trivially, so this
also yields a Baer--Fischer pair for that group.
\end{exmp}

The next two families of examples originate from disconnected algebraic groups
in odd characteristic, analogues of the characteristic~2
examples~\ref{ex:graph-transb} and~\ref{ex:ortho p=2b}.

\begin{exmp}   \label{exmp:graph-reflq}
We consider finite analogues of \cite[Ex.~4.1]{GM}. Let $k$ be algebraically
closed of odd characteristic and $G$ be the extension of $\GL_{2n}(k)$,
$n\ge2$, by a graph automorphism $y$ with centralizer $\Sp_{2n}(k)$.
Let $C_1$ be the class of an involution (in $G/Z([G,G])$) that is (up to
scalar) a pseudoreflection, $C_2$ the class of $y$. Taking fixed points under
a Steinberg endomorphism $F$ of $G$ we get Baer--Fischer pairs in $G^F$ of
type $\GL_n(q).2$ and $\GU_n(q).2$.
\end{exmp}

\begin{exmp}   \label{exmp:GOq}
We consider finite analogues of \cite[Ex.~4.2]{GM}. Let $k$ be algebraically
closed of characteristic not $2$, $G=\GO_{2n}(k)$, $n\ge4$, and $C_1$
containing elements with centralizer $\GL_n(k)$, $C_2$ containing reflections
in $G$. Let $F:G\rightarrow G$ be a Steinberg endomorphism. The stabilizer
$H\cong\GL_n(k)$ of a maximal isotropic subspace acts transitively on
non-degenerate 1-spaces of the natural module for $G$, with stabilizer a
maximal parabolic subgroup (which is connected). So if $H$ is chosen
$F$-stable, then $H^F$ acts transitively on $F$-stable non-degenerate 1-spaces,
whence we have a decomposition $G^F=G_v^F H^F$, for any $F$-stable
non-degenerate 1-space $v$. If $n$ is even, then this shows that for all odd
$q$ we get Baer--Fischer pairs in $\GO_{2n}^+(q)$ for classes with
centralizers $\GO_{2n-1}(q)$, together with $\GL_n(q)$ or $\GU_n(q)$, while
for odd $n$ we get such pairs in $\GO_{2n}^\pm(q)$ with centralizers
$\GO_{2n-1}(q)$, together with $\GL_n(q)$ in $\GO_{2n}^+(q)$,
respectively $\GU_n(q)$ in $\GO_{2n}^-(q)$.
\end{exmp}

Let's observe the following:

\begin{lem}   \label{lem:graph}
 Let $G$ be a finite group. Suppose that $C_1,C_2\subset G$ are conjugacy
 classes such that $x_1x_2$ has 2-power order for all $x_i\in C_i$.
 Let $\sigma$ be an automorphism of $G$ of order~2 interchanging $C_1$ and
 $C_2$, and set $\hat G$ the semidirect product of $G$ with $\sigma$. Then
 $C_1\cup C_2$, $[\sigma]$ is a Baer--Fischer pair in $\hat G$.
\end{lem}

\begin{proof}
Let $x\in C_1$, $y=\sigma$. Then $(xy)^2=xyxy=x\,x^\sigma
\in C_1C_2$ has 2-power order by assumption.
\end{proof}

This gives rise to two more families of examples.

\begin{exmp}   \label{exmp:b2f4}
Let $G=H.2$ where $H$ is either $F_4(2^{2m+1})$ or $\Sp_4(2^{2m+1})$,
the extension by the exceptional graph automorphisms of order~2.
Let $C_1\subset G\setminus H$ be a conjugacy class of outer involutions and
$C_2\subset H$ the $G$-conjugacy class of root elements of $H$ (note that
short and long root elements are fused in $G$). Let $x_1 \in C_1$ and let
$x_2$ be a short root element.
Then $x_2^{x_1}$ is a long root element.  By \cite[Ex.~6.3]{GMT},
$\langle x_2, x_2^{x_1}\rangle$ is $2$-group, whence $\langle x_1, x_2\rangle$
is.
\end{exmp}

\subsection{Baer--Fischer involution pairs in characteristic~3}
There exist further families of examples for groups of Lie type over the field
with three elements:

\begin{exmp}   \label{exmp:SLn(3)}
Let $G=\SL_{2n}(3).2$ (extension with a diagonal automorphism) with $n\ge 2$.
Let $c$ be an element with all eigenvalues $\pm i$ and let $d$ be a reflection.
We claim that $J:=\langle c,d\rangle$ is a 2-group. Indeed, let $v$ be an
eigenvector for the nontrivial eigenvalue of $d$ and consider the subspace
$W:=\langle v,cv \rangle$ of the natural module $V$ for $G$.
Since $c$ acts quadratically this space is invariant under $c$.  Any subspace
containing $v$ is $d$-invariant and so this space is $J$-invariant, and $J$
acts by a 2-group on it. Note that $J$ acts as a cyclic group of order $4$
on $V/W$. It suffices to show that $X=O_3(J) =1$. Suppose not and choose a
complement $W'$ to $W$ that is $c$-invariant. We can write
$$d: = \begin{pmatrix}  r  &   s   \\    0 & I_{n-2}\\   \end{pmatrix},$$
where $r$ is upper triangular and $s$ is a $2 \times(n-2)$ matrix. Since $d$
is a reflection, it follows that $s$ has only nonzero entries in the first row.
It follows that $[X,V]$ is $1$-dimensional. However, $c$ leaves invariant no
$1$-dimensional space, a contradiction. \par
The same construction applies to $\SU_{2n}(3).2$, $n\ge2$, with $c$ an element
of order~8 with minimal polynomial of degree~2 and $d$ a reflection.
\end{exmp}

\begin{exmp}   \label{exmp:O2n(3)}
Let $G=\CO_{2n}^\pm(3)$, a conformal orthogonal group in even dimension.
If $c,d\in\GL_{2n}(3)\cap G$ are as in the previous Example~\ref{exmp:SLn(3)},
they clearly also provide an example.

If $c$ preserves the orthogonal form, then in the algebraic group,
$c$  has a centralizer isomorphic to $\GL_n$ and so the centralizer
in the finite group is  $\GU_n(3)$ and so $G=\CO^+_{2n}(3)$ if $n$
is even and $G=\CO^-_{2n}(3)$ if $n$ is odd.

If $c$ does not preserve the orthogonal form, then in the algebraic
group, the eigenspaces for $c$ are nondegenerate spaces, whence
the centralizer is the normalizer of $\SO_n \times \SO_n$ and so in
the finite group, it will be $\SO^{\pm}_n(9)$.
\end{exmp}

\begin{exmp}   \label{exmp:O2n(3)ex2}
Let $G=\CO_{2n}^+(3)$, a conformal orthogonal group in even dimension.
Assume that $d$ is a reflection and $c$ has eigenspaces which are maximally
isotropic. In particular, the centralizer of $c$ is $\GL_n(3)$ and $c$ does
not preserve the form.

Let $v$ be a nonzero vector with $dv = - v$ and consider the subspace $W$
spanned by $v$ and $cv$.  Note that $W$ is $2$-dimensional since $v$ is not
an eigenvector for $c$. If $W$ is nondegenerate, clearly $\langle c,d \rangle$
is a $2$-group.

We claim that this is the case. For if not, we can choose an eigenvector $w$
for $c$ in $W$ that is not in the 1-dimensional radical. Then $w$ is
nonsingular  (for otherwise $W$ is totally singular which of course is not
the case); but all eigenvectors of $c$ are totally singular.
\end{exmp}

\begin{exmp}  \label{exmp:O2n(3)ex3}
Let $G=\CO_{2n}^{\pm}(3)$, a conformal orthogonal group with $n$ even.
Let $d$ be a bireflection. Let $Y$ be the $-1$ eigenspace of $d$. Rather than
consider the centralizer type of $d$, we consider the type of $Y$ (which
determines the centralizer of $d$ given the type of the entire space).

(i) Suppose that $c$ has centralizer $\GL_n(3)$. Thus, $c$ has eigenvalues
$\pm 1$ and the eigen\-spaces are totally singular (and $c$ does not preserve
the form). Let $V_1$ and $V_2$ be the eigenspaces for $c$. Suppose $v_i$,
$i =1,2$, are basis vectors for the $-1$ eigenspace of $d$. Write
$v_i = w_{1i} + w_{2i}$ where $w_{ji} \in V_j$. Note that $w_{1i}$ and $w_{2i}$
span a $2$-dimensional nondegenerate space of $+$ type.  Let $X$ be the span
of the $w_{ji}$.  If the span is $2$-dimensional, it is nondegenerate and
clearly $cd$ is a $2$-element.  If $\dim X =3$, then $X$ has a $1$-dimensional
radical and $c$ and $d$ have a common eigenvector. Choose another eigenvector
for $c$ that is not perpendicular to the radical of $X$ and this together with
$X$ span a $4$-dimensional nondegenerate $\langle c,d \rangle$-invariant space
(necessarily of $+$ type since $c$ acts on it).  Thus, we are reduced to the
case of $\CO_4^+(3)$.  So we see in this case that $c^Gd^G$ consists of
$2$-elements if and only if the $-1$ eigenspace has $-$ type, so if $d$ has
centralizer $\GO_{2n-2}^-(3)$.

(ii) Suppose that $c$ has centralizer $\GU_n(3)$. So $c$ has eigenvalues
$\pm i$ and the eigenspaces (over the algebraic closure) are totally singular.
Let $X$ be the subspace generated by $Y, cY$. Since $c$ is quadratic, we
see that $\dim X \le 4$. If $\dim X=2$, then clearly $cd=dc$ is a $2$-element.
Since $X$ is $c$-invariant,  $\dim X$ is even. So suppose that $\dim X = 4$.
Clearly, $X$ is not totally singular. If $X$ has a radical $R$, it would be
$2$-dimensional and $c$-invariant. Choose a totally singular $2$-dimensional
space $R'$ that is $c$-invariant with $R + R'$ nondegenerate. Then
$X + R'$ is a $6$-dimensional nondegenerate $\langle c, d \rangle$-invariant
space and we can apply the results for $\PSL_4$ and $\PSU_4$.
Finally, suppose that $X$ is nondegenerate.   The only $c$-invariant
$4$-dimensional nondegenerate spaces are of $+$ type. One easily computes that
$c^Gd^G$ consists of $2$-elements if and only if either $n$ is even and $d$ has
centralizer $\GO_{2n-2}^+(3)$, or $n$ is odd and $d$ has centralizer
$\GO_{2n-2}^-(3)$.
\end{exmp}

\begin{lem}   \label{lem:On(3)}
 Let $x, y$ be non-conjugate reflections in $G=\GO^{(\pm)}_n(q)$, $n \ge 3$,
 with $q$ odd. Then $\langle x,y^g\rangle$ is a $2$-group for all $g\in G$
 if and only if $q=3$.
\end{lem}

\begin{proof}
Clearly, $x, y$ are trivial on a common subspace of codimension 2.
If $n \ge 5$, this space cannot be totally singular and so by induction
we can pass to the orthogonal complement of this common space and so assume
that $n \le 4$.

Even for $n=4$, this space cannot be totally singular (because the subgroup
preserving a totally singular 2-space in a 4-space and trivial on the 2-space
is contained in the radical of some parabolic subgroup and so contains no
involutions).

So we see that it suffices to prove the result for $\GO_3(q)\cong\PGL_2(q)$;
one of the involutions is inner and the other outer. The normalizer of the
split and of the nonsplit torus are dihedral groups of order divisible by~4,
thus any element in their maximal cyclic subgroups of order $q\pm1$ is a
product of an inner by an outer involution.
Thus, $\langle x,y^g\rangle$ is always a 2-group if and only
if both $q+1$ and $q-1$ are powers of 2. The result follows.
\end{proof}

\begin{exmp}   \label{exmp:O(9)}
Let $G=\SO_{2n+1}(q).\langle\gamma\rangle$, where $q$ is an even power of an
odd prime and $\gamma$ the corresponding field automorphism of
$H=\SO_{2n+1}(q)$ of order~2. We claim that the class of reflections in $H$
with centralizer $\GO_{2n}^-(q)$ together with the class of $\gamma$ form a
Baer--Fischer pair when $q=9$. For this, let $V$ denote the natural
$2n+1$-dimensional
module for $H$ with invariant symmetric form $(\ ,\ )$. Note that $\gamma$
also acts naturally as a semilinear map on $V$. Let $x\in H$ be a reflection
and $v\in V$ an eigenvector for the non-trivial eigenvalue of $x$. If $\gamma$
stabilizes the 1-space generated by $v$, then it commutes with $x$ and thus
their product has order~2 as desired. Else, the 2-dimensional space
$W:=\langle v,\gamma(v)\rangle$ is invariant under
$\langle x,\gamma(x)\rangle$, and representing matrices are given by
$$x=\begin{pmatrix} -1&-2a\\ 0& 1\end{pmatrix},\quad
  \gamma(x)=\begin{pmatrix} 1&0\\ -2\gamma(a)& -1\end{pmatrix},$$
where $\sigma(v)=av+u$ for some $u\in\langle v\rangle^\perp$. Let $b=(v,v)$.
As $ab=(v,\gamma(v))=\gamma(\gamma(v),v)=\gamma(ab)$ we have that $ab$ lies
in the quadratic subfield. If $x$ has centralizer of minus type, then $b$ is
a non-square, so the same holds for $a$. But in this case, the two matrices
given above are seen to generate a group of order~8 when $q=9$. Now the Gram
matrix of $(\ ,\ )$ on $W$ is given by
$$\begin{pmatrix} b&ab\\ ab& \gamma(b)  \end{pmatrix},$$
so $W$ is non-degenerate. As $x,\gamma(x)$ act trivially on $W^\perp$,
the claim follows.
\end{exmp}

\subsection{The classification of Baer--Fischer involution pairs}

\begin{thm}   \label{thm:p=2}
 Let $G$ be a finite almost simple group. Suppose that $c, d \in G$ are
 involutions such that $\langle c, d^g \rangle$ is a $2$-group for all
 $g \in G$.  Then one of the following holds (up to order):
 \begin{enumerate}[\rm(1)]
  \item $G$ is a finite group of Lie type in characteristic $2$, and $c$, $d$
   are unipotent elements as in \cite[Thm.~4.6]{GMT}; more specifically
  \begin{enumerate}[\rm(a)]
   \item $G=\Sp_{2n}(2^a)$, $n\ge2$, with $c,d$ as in \cite[Thm.~4.6(1)]{GMT};
   \item $G=F_4(2^a)$, with $c,d$ as in \cite[Thm.~4.6(2)]{GMT};
  \end{enumerate}
  \item $F^*(G)$ is a finite group of Lie type in characteristic $2$,
   $c$ is unipotent and $d$ a graph automorphism as in \cite[Thm.~3.7(2)]{GM};
   more specifically
  \begin{enumerate}[\rm(a)]
   \item $F^*(G)=\PSL_{2n}(2^a)$ or $\PSU_{2n}(2^a)$, $n\ge2$, with $c,d$ as in
    \cite[Thm.~3.7(2)(c)]{GM};
   \item $F^*(G)=\OO_{2n}^{\pm}(2^a)$, $n\ge3$, with $c,d$ as in
    \cite[Thm.~3.7(2)(d)]{GM};
   \item $F^*(G)=E_6(2^a)$ or $\tw2E_6(2^a)$, with $c,d$ as in
    \cite[Thm.~3.7(2)(e)]{GM};
  \end{enumerate}
  \item $F^*(G)=\Sp_4(2^{2m+1})'$ or $F_4(2^{2m+1})$, $m\ge0$, $c$ is a long
   root element and $d$ is a graph automorphism;
  \item $G$ is a disconnected finite group of Lie type in odd characteristic,
   and $c$ and $d$ are as in \cite[Thm.~4.5]{GM}; more specifically
  \begin{enumerate}[\rm(a)]
   \item $F^*(G)=\PSL_{2n}(q)$ or $\PSU_{2n}(q)$, $n\ge2$, with $q$ odd,
    $c$ is a pseudo-reflection (modulo scalars) and $d$ a graph automorphism
    with centralizer $\PSp_{2n}(q)$;
   \item $F^*(G)=\OO_{2n}^+(q)$, $n\ge4$ even, with $q$ odd, where $c$ has
    centralizer $\GL_n^\pm(q)$ and $d$ is a graph automorphism
    with centralizer $\OO_{2n-1}(q)$;
   \item $F^*(G)=\OO_{2n}^{\pm}(q)$, $n\ge5$ odd, with $q$ odd, where $c$ is an
    involution with centralizer $\GL_n^\pm(q)$ and $d$ is a graph automorphism
    with centralizer $\OO_{2n-1}(q)$;
  \end{enumerate}
  \item $F^*(G)$ is a finite group of Lie type in characteristic~3, more
   specifically
  \begin{enumerate}[\rm(a)]
   \item $F^*(G)=\PSL_{2n}(3)$ or $\PSU_{2n}(3)$, $n\ge2$ and $c$ lifts to an
   element of $\SL_{2n}(3)$ with eigenvalues $\pm i$ and $d$ is a reflection;
   \item $F^*(G)=\OO_n^{\pm}(3)$ with $n\ge6$, where $c$ and $d$ are
    non-conjugate reflections;
   \item $F^*(G)=\OO_{2n}^\pm(3)$, $n\ge4$, where $c$ is a reflection and $d$
    has centralizer $\OO_n^{(\pm)}(9)$;
   \item $F^*(G)=\OO_{2n}^+(3)$, $n\ge4$ even, where $c$ has centralizer
    $\OO_{2n-2}^\pm(3)$ and $d$ has centralizer $\GL_n^\mp(3)$;
   \item $F^*(G)=\OO_{2n}^\pm(3)$, $n\ge5$ odd, where $c$ has centralizer
    $\OO_{2n-2}^-(3)$ and $d$ has centralizer $\GL_n^\pm(3)$;
   \item $F^*(G)=\OO_{2n+1}(9)$, $n\ge2$, $c$ is a reflection with centralizer
    $\OO_{2n}^-(9)$ and $d$ is a field automorphism;
  \end{enumerate}
  \item $G=\fS_{2n}$, $n\ge3$, $c$ is a fixed point free involution and $d$
   is a transposition; or
  \item $G=Fi_{22}.2$, $c$ is an inner $3$-transposition and $d$ an outer
   involution with centralizer $\OO_8^+(2):\fS_3$.
 \end{enumerate}
\end{thm}

Here, $\GL_n^+(q)$ denotes $\GL_n(q)$, and $\GL_n^-(q)=\GU_n(q)$.
We split up the proof of the claim into a series of proposition.

\begin{prop}
 Theorem~\ref{thm:p=2} holds when $S=F^*(G)$ is not of Lie type.
\end{prop}

\begin{proof}
For $S$ sporadic or $\tw2F_4(2)'$, a check with the known character
tables shows that the only example occurs in $\Aut(Fi_{22})$. (See also
\cite[Lemma~8.3]{GMT} for the case when $G=S$.)\par

If $S= \fA_6$ and $G$ is not contained in $\fS_6$, then it is easily checked
that $\{2b,2c\}$ is the only possible pair (notation as in {\sf GAP}). This
occurs in~(3) for $m=0$.

If $S \cong \fA_n$, $n\ge5$, and $G \le \fS_n$, then we claim the only
possibility is that (up to order) $c$ is fixed point free and $d$ is a
transposition, as in~(6), which clearly is an example. If $c$ and $d$ both
have fixed points, then the result holds by
induction (starting with $n=5$). So we may assume that $c$ is a fixed point
free involution. Suppose that $d$ is not a transposition. Then $c$ and $d$
both leave invariant a subset of size $6$ with $d$ not acting as a
transposition. It is straightforward to see for all possibilities that
$\langle c,d^g \rangle$ can generate a subgroup of order divisible by $3$.
\end{proof}

\begin{prop}
 Theorem~\ref{thm:p=2} holds when $S=F^*(G)$ is of Lie type in
 characteristic~$2$.
\end{prop}

\begin{proof}
If $c, d$ are both inner, the result follows by \cite[Thm.~4.6]{GMT}.
Thus $S$ is not a Suzuki or Ree group and we may assume that $d$ induces an
outer automorphism, i.e., either a graph, field or graph-field automorphism
(notation as in \cite[2.5.13]{GLS}).
Note that groups in characteristic~$2$ do not have outer diagonal
automorphisms of even order.

We first deal with the exceptional graph-field automorphisms of $B_2(2^a)$
and $F_4(2^a)$. When $a$ is odd, all involution classes have representatives in
$\Aut(B_2(2))=\fS_6$ respectively $\Aut(F_4(2))$, and direct calculation shows
that the only examples in the latter two cases are those in~(3) as in
Example~\ref{exmp:b2f4}. If $a$ is even, we use that there are no cases in
$\Sp_4(4).2$, and that all involution classes in $F_4(2^a).2$ contain
representatives in the subsystem subgroup $\Sp_4(2^a).2$ to see that no
new examples arise.

Next assume that $d$ induces a field or graph-field automorphism (in
particular, $S$ is not twisted of degree~2). If $S$ is one of $\PSL_n(2^{2f})$,
$2\le n\le 4$, then all unipotent classes of $\Aut(S)$ have representatives
in $\Aut(\PSL_n(4))$, and direct calculation shows that there are no examples.
Otherwise, let $P$ be an end node parabolic subgroup of $S$ stable under any
graph automorphism of order~2, respectively one of type $\GL_{n-2}$ in
$\PSL_n(q)$, and such that it contains conjugates of $c$ outside its unipotent
radical. Then $N_{\Aut(S)}(P)S=\Aut(S)$, and $d$ acts by field or graph-field
automorphisms on the simple Levi factor of $P$. Hence there are no examples
by induction.

Thus we may suppose that $d$ is a graph automorphism, so $S$ is of (possibly
twisted) type $A_n$, $D_n$ or $E_6$. First consider $S=\PSL_n(q)$. By direct
calculation there are no examples in $\Aut(\PSL_3(2))$, and the only
possibility in $\Aut(\PSL_4(2))=\fS_8$ is that $c$ is a transvection and $d$
is a graph automorphism with centralizer $\Sp_4(2)$, as in~(2a) of
the conclusion. Now for $n\ge5$ we may again reduce to a parabolic subgroup $P$
of type $\GL_{n-2}$ normalized by suitable conjugates of $c$ and $d$. Hence by
induction there are no examples when $n$ is odd, and when $n$ is even the
image of $c$ in $N_G(P)/O_2(P)$ must be a transvection and $d$ a graph
automorphism with centralizer $\Sp_n(q)$. If $c$ is not a transvection
we may arrange that its image in $P/O_2(P)$ is neither. So we only get
case~(2a) which is an example by \cite[Thm.~3.7]{GM}.

Next assume that $S=\PSU_n(q)$. Again by direct calculation there are no
examples in $\PSU_3(2).2=3^{1+2}.2.\fS_4$, while for $\Aut(\PSU_4(2))$ we
only find the case in assertion~(2a). We can now argue by induction exactly as
in the previous case.

Now let $S=\OO_{2n}^+(q)$, $n\ge4$. By the previous paragraphs, for
$\OO_6^+(q)=\PSL_4(q)$ we just have the example in~(2a). By descending to the
parabolic subgroup of type $\OO_{2n-2}^+(q)$ we see that this is the only
case for $\OO_{2n}^+(q)$, leading to~(2b) by Example~\ref{ex:ortho p=2b}.
The same inductive argument works for $S=\OO_{2n}^-(q)$, starting at
$\OO_6^-(q)=\PSU_4(q)$. Finally, for $S=E_6(q)$ note that the subsystem
subgroup $A_5(q)$ contains
representatives from all three inner involution classes and is stabilized by
the graph automorphism of order~2. Since the only example for $A_5(q)$ arises
from graph automorphisms (see above), we can only get an example for $S$ when
one class is the inner class of type $A_1$, and the other contains the graph
automorphism. This actually occurs by Example~\ref{exmp:E6b}. Similarly, for
$S=\tw2E_6(q)$ we may descend to the subsystem subgroup
$\tw2A_5(q)$ to arrive at~(2c).
\end{proof}

\begin{prop}   \label{prop:classodd}
 Theorem~\ref{thm:p=2} holds when $S=F^*(G)$ is of classical Lie type in
 odd characteristic $p$, but not an even-dimensional orthogonal group.
\end{prop}

\begin{proof}
First suppose that $F^*(G) = \PSp_{2n}(q)$ with $n\ge1$ and $q$ odd. If $n=1$
(and $q\ge5$)
an elementary calculation shows that inner diagonal involutions do not lead to
examples. Next assume that $G$ involves field automorphisms (so $q=q_0^2\ge9$).
If there are two such classes, then $[G:F^*(G)]=2$ and by direct matrix
calculation we find products which are not of 2-power order. If just one
of the two classes contains field automorphisms, again a direct calculation
shows that necessarily $q_0^2-1$ must be a 2-power, so the only example occurs
for $S=\PSL_2(9)=\fA_6$, a case already discussed. \par
So now suppose that $S=\PSp_{2n}(q)$ with $n\ge2$.
All involution classes have representatives normalizing but not centralizing
a Levi subgroup of type $\PSp_2(q)=\PSL_2(q)$. Thus by the previous case we
can only get examples when $q\in\{3,9\}$. In these cases, we reduce to
$\PSp_4(q)$ or $\PSp_6(q)$ instead. The
possibilities for $\PSp_4(3)=\PSU_4(2)$ have already been discussed. Explicit
computation of structure constants shows that there are no cases for
$\PSp_6(3)$. Since the only examples for $\PSL_2(9)$ involve field
automorphisms, the same must be true for $\PSp_{2n}(9)$. For $\PSp_4(9)$
explicit computation of structure constants in \GAP{} \cite{GAP} yields only
case~(5f).
Finally, for $\PSp_6(9)$ one class
must contain field automorphisms, as for $\PSp_4(9)$, which is hence uniquely
determined by \cite[Prop.~4.9.1]{GLS}, and the other must contain diagonal
automorphisms, with all Jordan blocks of size~2. All such involution classes
normalize the extension field subgroup $\PSL_2(9^3)$, whence we get no further
example. (Alternatively, a direct computation with \GAP{} gives the claim.)

Now assume that $S=\PSL_n(q)$ with $n\ge 3$. Note that any inner diagonal
involution can either be lifted to an involution in $\GL_n(q)$, or to an
element with all Jordan blocks of size~2. Thus if $c, d$ are both inner
diagonal, we can reduce to the case of
$\PGL_2(q)$, whence there are no examples for $q\ne3$ (as the examples for
$\PSL_2(9)$ involve field automorphisms). Similarly if $q=3$, unless all
eigenvalues for $c$ (interchanging $c$ and $d$ if necessary) are $\pm i$, we
may reduce to $\PSL_3(3)$, for which no example occurs. In particular,
$n$ is even. We claim that $d$ must be a reflection (modulo scalars).
Since $d$ cannot be conjugate to $c$ (by the standard Baer--Suzuki theorem),
it follows that $d$ has all eigenvalues $\pm 1$. If $d$ is not a reflection,
then we can reduce to $\PSL_4(3)$ and see from the character table that there
are no examples, while if $d$ is a reflection, we get case~(5a) by
Example~\ref{exmp:SLn(3)}.

Next suppose that $d$ is an outer involution. If $d$ is a field or graph-field
automorphism (so in particular $q\ge9$), we reduce to the case $n=2$ (when
$q \ne 9$) or $n=3$ and it is straightforward to compute that there are no
examples, while for $q=9$ and $n$ even, we may reduce to
$\PSL_2(81)\le\PSL_4(9)$ for which we already saw that no example exists.

So suppose that $d$ is a graph automorphism. If $c$ is also a graph
automorphism, then $n$ is even since for odd $n$ there is only one class.
We may reduce to $\PSL_4(q)$, in which case it is easy to write down
representatives for all four classes such that products do not have 2-power
order, except when $q=3$. For $q=3$, a direct check shows that only the
example in~(5b) is possible (see Lemma~\ref{lem:On(3)}).

So $c$ is an inner diagonal involution. If $n$ is odd, then we can reduce
to the case of $n=3$ where the result is straightforward to verify.
Similarly, when $n$ is even we may reduce to the case that $n=2$ to see that
$q\in\{3,9\}$. In that case, we reduce to $n=4$ where the only examples are
those in~(4a), see Example~\ref{exmp:graph-reflq}.

Now let $S=\PSU_n(q)$ with $n\ge3$. Again any inner diagonal involution can
either be lifted
to an involution in $\GU_n(q)$, or to an element with all Jordan blocks of
size~2. Thus if $c,d$ are both inner diagonal, we can reduce to the case of
$\PGL_2(q)$, whence there are no examples for $q\ne3$. Similarly if $q=3$,
unless all Jordan blocks for $c$, say, have size~$2$, we may reduce to
$\PSU_3(3)$, for which no example occurs. In particular, $n$ is even. As in
the linear group case, $d$ must be a reflection (modulo scalars), in which
case we get case~(5a) by Example~\ref{exmp:SLn(3)}.
Next suppose that $d$ is an involution which is not inner-diagonal, hence a
graph automorphism. We now argue as for the case of $\PSL_n(q)$ to arrive at
the cases~(4a).

Finally, assume that $S = \OO_{2n+1}(q)$ with $n\ge3$. Again, we may reduce to
a Levi subgroup of type $\OO_{2n-1}(q)$. Note that for $\OO_5(q)=\PSp_4(q)$
we saw that $q\in\{3,9\}$ and, for $q=9$, one class consists of reflections
with centralizer of minus type, the other of field automorphisms. The latter
gives case~(5f) by Example~\ref{exmp:O(9)}. For $\OO_7(3)$ explicit computation
shows that the only example is as in case~(5b), by Lemma~\ref{lem:On(3)}.
\end{proof}

\begin{prop}
 Theorem~\ref{thm:p=2} holds when $F^*(G)=\OO_{2n}^\pm(q)$, $n\ge3$,
 with $q=p^a\ne3$ odd.
\end{prop}

\begin{proof}
Note that the cases for $\OO_6^+(q)=\PSL_4(q)$ and $\OO_6^-(q)=\PSU_4(q)$ were
classified in Proposition~\ref{prop:classodd}. So let $S=\OO_{2n}^\pm(q)$ with
$n\ge4$. If both classes contain field automorphisms, they may be chosen to
normalize a Levi subgroup of type $\OO_6^\pm(q)$. Since there are no examples
with
field automorphisms for $\OO_6^\pm(q)$, these cannot occur for $S$ either.
All other classes of involutions have representatives in the conformal
orthogonal group $\CO_{2n}^\pm(q)$, and they contain elements normalizing,
but not centralizing a Levi subgroup of type $\OO_{2n-2}^\pm(q)$. Thus, even
with just one class containing field automorphisms, we do not get examples.

We may hence assume we are inside $\CO_{2n}^\pm(q)$, $n\ge4$.

We claim that the only examples are when one of the classes are reflections
and the other is as given in (4b) and~(4c). By Example~\ref{exmp:GOq}, the
cases listed in the theorem are in fact examples. Note that any involution
leaves invariant a $4$-dimensional non degenerate space of $+$ type. In
particular,
we can reduce to the case that $2n=6$  or $8$ (since starting from any
case not allowed in the theorem, we can peel of $4$-dimensional nondegenerate
spaces in such a way that the pair is still not as in the theorem).  If $2n=6$,
we are done by appealing to the results for $\PSU_4$ and $\PSL_4$.  If $2n=8$,
the same argument with $2$-dimensional nondegenerate spaces works unless
the elements do not leave invariant a $2$-dimensional nondegenerate space
of the same type.  This only happens when one of the involutions has
eigenvalues $\pm 1$ and totally singular eigenspaces and the other element
has quadratic minimal polynomial.  Thus, we are in $\CO_{8}^+(q)$.
In this case each element acts nontrivially
on a totally singular $4$-dimensional space and the result then follows by
the $L_4$ result.
\end{proof}

In order to deal with the case $q=3$, we first describe the relevant classes
of involutions.

\begin{lem}   \label{lem:invclasses}
 Let $C$ be a conjugacy class of $\CO_{2n}^\pm(3)$, $n\ge3$, containing
 involutions modulo the center. Then one of the following holds:
\begin{enumerate}[\rm(1)]
 \item $C$ is contained in $\GO_{2n}^\pm(3)$ and consists of elements with
  eigenvalues
  $\pm 1$. Then we may assume that the $-1$ eigenspace is  $e$-dimensional
  with $e \le n$. The centralizer in the algebraic group is
  $\GO_e \times \GO_{2n-e}$. There are two classes depending upon the type
  of the $-1$ eigenspace.
 \item $C$ is not in $\GO_{2n}^\pm(3)$ and consists of elements with eigenvalues
  $\pm 1$. It follows that the eigenspaces for $C$ are maximal totally
  singular. The centralizer is $\GL_n(3)$ and so this only occurs in $+$ type.
 \item $C$ is contained in $\GO_{2n}^\pm(3)$ and has eigenvalues $\pm i$. Thus,
  $C$ lies in $\GO_{2n}^+(3)$ if $n$ is even, in $\GO_{2n}^-(3)$ if $n$ is odd.
  The centralizer in the algebraic group $\SO_{2n}$ is $\GL_n$ since the
  eigenspaces are totally singular. Thus, the centralizer is $\GU_n(3)$.
 \item $C$ is not contained in $\GO_{2n}^\pm(3)$ and has eigenvalues $\pm i$.
  In the
  algebraic group, the eigenspaces are nondegenerate of dimension~$n$ and so
  the centralizer is $\GO_n \times \GO_n$ --- in $\SO_{2n}$ the centralizer
  is a subgroup of index~$2$ and so there are two such classes in
  $\SO_{2n}^\pm(3)$ and also in $\GO_{2n}^\pm(3)$, with centralizer
  $\GO_n^{(\pm)}(9)$.
 \end{enumerate}
\end{lem}

\begin{prop}
 Theorem~\ref{thm:p=2} holds when $F^*(G)=\OO_{2n}^\pm(3)$, $n\ge3$.
\end{prop}

\begin{proof}
Note that there are no field automorphisms so we are inside $\CO_{2n}^\pm(3)$.
We deal with the various possibilities.

\smallskip
A. First suppose $C$ and $D$ are both as in (1) of Lemma~\ref{lem:invclasses}.
\smallskip

If $C,D$ consist of nonconjugate reflections, we get case~(5b) by
Lemma~\ref{lem:On(3)}. Else, if $d \in D$ has an eigenvector $v$ with $dv=-v$
and $c \in C$ with $cw=-w$ with $v,w$ of the same norm, then we can choose a
nondegenerate $5$-space which is $c,d$ invariant (replacing by conjugates if
necessary) and check in $\GO_5(3)$ that $cd$ need not be a $2$-element.

The remaining case here is when all eigenvectors of $c \in C$ and $d \in D$
have distinct norms but at least one is not a reflection.  Again, we can find
a nondegenerate $5$-space where this happens and so there are no examples.

\smallskip
B. Next suppose that $C$ is as in (1) and $e \ge 3$.
\smallskip

We claim there are no such examples.  Note that any element $d$ in one of
the conjugacy classes in (2), (3) or (4) will preserve a nondegenerate
$8$-dimensional space of some type. Now choose $c$ preserving the same type
of $8$-space with the $-1$ eigenspace of dimension $3$ or $4$.  One computes
in $\GO_8^{\pm}(3)$ to see that it is not always the case that
$cd$ is a $2$-element.

\smallskip
C.  Suppose that $C$ is a reflection.
\smallskip

(i)  Suppose that $D$ is as in (2). Let $c \in C$ and let $U_1,U_2$ be the
eigenspaces for $d \in D$. Let $v$ be an eigenvector of $c$ with $cv=-v$ and
write $v = u_1 + u_2$ with $u_i \in U_i$.
Then we see that $v$ is contained in a $4$-dimensional nondegenerate invariant
subspace for $d$ (and necessarily $c$-invariant as well since it contains
the $-1$ eigenspace.)   Since the eigenspaces for $d$ are totally singular,
we are in $\GO_4^+(3)$ and it is easy to see that $cd$ is a $2$-element. These
are cases~(4a) and~(4b).

(ii)  Suppose that $D$ is as in (3) or (4).

Let again $v$ be such that $cv=-v$.  If $d \in D$, then $\langle v, dv
\rangle$ is $2$-dimensional and must be nondegenerate (since it is not
totally singular and $d$ acts irreducibly).  Thus
$\langle c, d \rangle < \CO_2^-(3) \times \CO_{2n-2}(3)$ and $c$ is trivial on
the $2n-2$ space. Computing in the $2$-group $\CO_2^-(3)$ shows this is an
example, giving (5c).

D.  Suppose that $C$ consists of bireflections.  There are two classes
of such differentiated by whether the $-1$ eigenspace has $+$ type or
$-$ type  (this determines the centralizer but this invariant does not
change when passing to a nondegenerate space containing the $-1$ eigenspace
of the bireflection). Note that if $2n > 8$, then $c$ acts as a bireflection
on $6$- or $8$-dimensional nondegenerate spaces of either type.

We have already taken care of $D$ as in (1).  So assume that $D$ consists
of elements as described in (2), (3) or (4).  Note that $d$ will preserve
a $4$-dimensional nondegenerate space of $+$ type in all cases.

It follows by Example~\ref{exmp:O2n(3)ex3} that the cases listed in the
theorem do occur. Moreover, the only possible choices for $D$ are as given in
the theorem: As we have already noted all involutions preserve a $4$-dimensional
nondegenerate space of $+$ type and so arguing as above, we reduce
to the case $2n=6$ or $8$ where we compute that $c^Gd^G$ does not consist
of $2$-elements.

E.  Neither $C$ nor $D$ is as in (1)

Now we want to show there are no examples. We argue precisely as in the last
paragraph of case D to reduce to the cases $2n=6$ or $8$ and compute that
$c^Gd^G$ does not consist of $2$-elements.
\end{proof}

Before treating the remaining cases, let's observe the following, which can
easily be deduced using \cite{Chv}:

\begin{lem}   \label{lem:exc2}
 Let $G$ be an exceptional group of adjoint Lie type of rank at least~4 in
 odd characteristic. Then all conjugacy classes of involutions have
 (non-central) representatives in a natural subgroup $H$ as listed in
 Table~\ref{tab:exc2}, where again $T$ denotes a 1-dimensional split torus.
\end{lem}

\begin{table}[htbp]
 \caption{Subgroups intersecting all involution classes}   \label{tab:exc2}
\[\begin{array}{|c|lllll|}
\hline
   G& F_4(q)& E_6(q)_\ad& \tw2E_6(q)_\ad& E_7(q)_\ad& E_8(q)\\
\hline
   H& B_4(q)& F_4(q)& F_4(q)& D_6(q)T& D_8(q)\\
\hline
\end{array}\]
\end{table}

\begin{proof}
For the groups $F_4(q)$ and $E_8(q)$ all involution classes already have
representatives in a maximally split torus, a conjugate of which is contained
inside $H$. For the remaining types, the involution classes in $H$, the
component group of their centralizers and their fusion into $G$ can be
computed using the relevant \Chevie-commands. The claim follows by
inspection.
\end{proof}

\begin{prop}
 Theorem~\ref{thm:p=2} holds when $S=F^*(G)$ is of exceptional Lie type in
 odd characteristic $p$.
\end{prop}

\begin{proof}
If one of $c$, $d$ induces a field automorphism on $F^*(G)$, then we may
reduce to its fixed point group by the standard argument, using that all
field automorphisms of order~2 are conjugate (see \cite[Prop.~4.9.1(d)]{GLS})
and that the centralizer contains representatives from all involution
classes. \par
If $c$ induces a graph automorphism, then for $F^*(G)=G_2(3^{2m+1})$ we
check explicitly in $G_2(3)$. Now assume that $F^*(G)={}^{(2)}E_6(q)$. There
are two classes of graph automorphisms, with centralizer $F_4(q)$ respectively
$C_4(q)$ in $F^*(G)$. Both contain representatives from both classes of inner
involutions. When $C_{F^*(G)}(c)=F_4(q)$ then $F_4(q)\times\langle c\rangle$
also contains non-central conjugates of $c$, since $F_4(q)$ has involutions
with centralizer $B_4(q)$. Thus by induction we do not get examples. A similar
argument applies when $C_{F^*(G)}(c)=C_4(q)$.   \par
The groups $^2G_2(3^{2m+1}),G_2(q),\tw3D_4(q)$ have a single class of
involutions. For the groups of rank at least~4 and pairs of inner-diagonal
elements we use Lemma~\ref{lem:exc2} and induction. We are thus left with
$^{(2)}E_6(q)$ and at least one of $c,d$ a graph-field automorphism. Then
again by \cite[Prop.~4.9.1(d)]{GLS} we can reduce to the centralizer
$F_4(q)$ where there are no examples by the above.
\end{proof}

This completes the discussion of all cases and hence the proof of
Theorem~\ref{thm:p=2}.

\begin{exmp}
If in Theorem~\ref{thm:p=2} we allow arbitrary classes of non-trivial
2-elements, there will be additional examples. In case~(3) we may take
pseudo-reflections of arbitrary 2-power order; and for $\PSU_{2n}(3)$
with $n$ odd in case~(5a), we get examples with one class containing elements
of order~4. In addition to those, we are aware of examples in several further
simple groups $S$. These are $S=\fA_6$ with the following six pairs
$$\{2a,4b\},\{2a,8a\},\{2a,8b\},\{4a,4b\},\{4a,8a\},\{4a,8b\}$$
for $M_{10}=\fA_6.2_3$ (notation as in {\sf GAP}) and
$$\{2b,8a\},\{4b,2c\},\{4b,8a\}$$
for $\Aut(\fA_6)$; several pairs of classes in $\PSL_2(81).4$ with products of
orders 8 and~16; a pair of classes of elements of orders 2 and~8 in
$\PSL_3(4).2^2$ for which all products have order either~4 or~8; and a pair of
classes  of elements of orders~2 and~4 in $\PSU_4(3).D_8$ for which the
product is a single class of elements of order~8. \par
There are no such examples in the symmetric group $\fS_n$: Since a
transposition is not the square of any element, we may assume that $c$ is a
transposition.
Suppose that $d^2$ is a fixed point free involution.  So $n \ge 8$ and it
suffices to consider that case. Indeed, we can reduce to the case $\fS_4$
with $d$ a $4$-cycle. We see that $\langle c, d^g \rangle$ can be $\fS_4$,
whence the claim.
\end{exmp}

\section{Commutators}  \label{sec:commut}

In this section we will prove the main case of Theorem~\ref{thm:mainB}.
If $x,y$ are elements of a group $G$, let $[x,y]=x^{-1}y^{-1}xy$ denote the
commutator of $x$ and $y$.

\begin{thm}   \label{thm:main1}
 Let $p$ be a prime and $G$ a finite group. Let $C$ be a normal subset of $G$
 consisting of $p$-elements. If $p=5$, assume that $C$ is closed under
 squaring. If $C$ is closed under taking commutators, then
 $\langle C \rangle \le O_p(G)$.
\end{thm}

For the proof we proceed in a series of lemmas. Let $G$ be a counterexample
with $|G| +|C|$ minimal. Clearly
$G = \langle C \rangle$ and $O_p(G)=1$. Moreover, we may assume that every
element of $C$ is a commutator of elements in $C$ and so $G$ is perfect
(otherwise replace $C$ by the set $D$ of commutators of pairs of elements
in $C$; then by minimality $D$ generates a $p$-group and so $D$ is trivial,
whence the group is abelian).

\begin{lem}   \label{lem:inverses}
 $C$ is closed under inverses.
\end{lem}

\begin{proof}
We have $[x,y][y,x]=1$ for all $x,y\in C$.
\end{proof}

\begin{lem}   \label{lem:unique}
 Each $x\in C$ lies in a unique maximal subgroup $M$
 and $C \cap M \subseteq O_p(M)$.
\end{lem}

\begin{proof}
Since $x \in O_p(M)$ for each maximal subgroup $M$ containing $x$, we have that
$\langle x \rangle$ is subnormal in $M$. By a result of Wielandt \cite{W}
this implies that either $M$ is unique or $\langle x \rangle$ is subnormal
in $G$, whence $x \in O_p(G)$, a contradiction.
\end{proof}

\begin{lem}   \label{lem:simple}
$G$ is simple.
\end{lem}

\begin{proof}
Suppose not. Let $N$ be a minimal normal subgroup.
By induction, $G/N$ is a $p$-group but also perfect, whence $G=N$ is simple.
\end{proof}

\begin{lem}   \label{lem:p odd}
$p \ne 2$.
\end{lem}

\begin{proof}
Let $G$ be a minimal counterexample. Let $P$ be a Sylow $2$-subgroup of $G$.
By Lemma~\ref{lem:unique} $P$ is contained in a unique maximal subgroup $M$
and $C \cap P \subseteq O_2(M)$.

In \cite[Thm.~A]{As1}, Aschbacher classifies all almost simple groups
in which a Sylow $2$-subgroup is contained in a unique maximal subgroup $M$.
Thus, $C \cap O_2(M)$ is nonempty and inspection of the conclusion of
Aschbacher's theorem leave only the following cases (for the case $G$ is simple):

\begin{enumerate}
 \item $G$ is a rank~1 Lie type group in characteristic $2$ (i.e.,
  one of $\PSL_2(q), \PSU_3(q)$ or $\tw2B_2(q^2)$);
 \item $G=\PSL_2(q)$ with $q$ odd; or
 \item $G$ is of Lie type in odd characteristic and $O_2(M)$ has exponent $2$.
 \end{enumerate}

In the first case, by inspection of the Sylow $2$-subgroup, we see that
$C$ must contain involutions. By the Baer--Suzuki theorem, any involution
of a simple group is contained in a dihedral group of order twice an odd
number, a contradiction.   Similarly, $C$ must consist of involutions
in the third case and we obtain a contradiction.

So we are reduced to the case of $\PSL_2(q)$ with $q \ge 5$ odd.
It is straightforward to compute in these cases that some commutator
of elements in $C$ is not a $2$-element. Indeed, if $q \equiv 1 \pmod 4$, then
we can choose a non-commuting pair of elements in $C$ contained in a
Borel subgroup and so the commutator will be a nontrivial unipotent element.
Else, first suppose that $q$ is not a power of $3$. Take $x \in C$ with
$$  x = \begin{pmatrix}   0  &   1  \\   -1  & t \\ \end{pmatrix} $$
with $t$ to be the trace of an element in $C$. Take $y$ conjugate to $x$ of
the form
$$ y= \begin{pmatrix}   t  &   1  \\   -1 & 0 \\ \end{pmatrix}. $$
Then $\tr(xyx^{-1}y^{-1}) = t +3$. Since $C$ is closed under taking
commutators, we see that the traces of elements in $C$ can take on any value
and so in particular the value $\pm 1$ (corresponding to elements
of order $3$). 

Finally, consider the case that $q=3^a \ge 27$ with $a$ odd (as
$q\not\equiv1\pmod4$). Then the Sylow $2$-subgroups of $\PSL_2(q)$ are
elementary abelian of order $4$. Then $C$ contains involutions and the
result follows by the Baer--Suzuki theorem.
\end{proof}

\begin{lem}  \label{lem:nci}
 If  $x \in C$ and $h \in G$ are nontrivial, then $[x,x^{-g}]\ne 1$ for
 some conjugate $g$ of $h$.  
\end{lem}

\begin{proof}
Let $M$ be the unique maximal subgroup containing $x$. Of course,
$M$ contains $C_G(x)$. So if the result is false, $x^{-g} \in C_G(x) \le M$ for
all conjugates $g$ of $h$. Then $M$ is also the unique maximal subgroup
containing $x^{-g}$, but of course $x^{-g} \in M^g$.
Thus, $M=M^g$ and so $M$ is normalized by all conjugates of $h$ and so by $G$;
a contradiction.
\end{proof}

\begin{lem}  \label{lem:4}
 Let $g\in G$.
 \begin{enumerate}
  \item[(a)] Some nontrivial element of $C$ is a product of $4$ elements
   which are conjugate to either $g$ or $g^{-1}$.
  \item[(b)] If $g$ is an involution, then $g$ inverts some nontrivial element
   of $C$ and this element is the product of two conjugates of $g$.
 \end{enumerate}
\end{lem}

\begin{proof}
By the previous result, we can choose $x \in C$ so that
$$ 1 \ne [x,x^{-g}]= x^{-1}gxg^{-1}xgx^{-1}g^{-1} \in C.$$
If g is an involution, then this becomes
$[x,x^{-g}] = x^{-1}gx\cdot (gx) g (gx)^{-1}$.
Thus $x^{-1}gx$ inverts $[x,x^{-g}]$, and hence $g$ inverts
$x[x,x^{-g}]x^{-1}\in C$.
\end{proof}

\begin{lem}   \label{lem:alt}
 $ G \ne \fA_n$, $n \ge 5$.
\end{lem}

\begin{proof}
Note that $p\ne2$ by Lemma~\ref{lem:p odd}.
Let $g \in \fA_n$ be a product of two transpositions. Then the only
$p$-elements inverted by $g$ are of order at most $5$ and move at most $6$
points. So $p \le 5$ by Lemma~\ref{lem:4} and $C$ contains a $p$-cycle, or
$p=3$ and $C$ contains elements that are products of two $3$-cycles.

If $p=3$, commutators of elements in $C$ can be nontrivial involutions
(by considering $\fA_4$, and in $\fA_6$, we can apply an automorphism and
reduce to $3$-cycles).
If $p=5$, then since $C$ is closed under squaring $C$ contains all $5$-cycles
and a straightforward computation (in $\fA_5$) shows that there are commutators
in $C$ that have order prime to $5$.
\end{proof}

\begin{lem}   \label{lem:natural}
 $G$ is not a group of Lie type in characteristic $p$.
\end{lem}

\begin{proof}
First suppose that $G$ has rank $1$.
Then (as $p \ne 2$), $G=\PSL_2(q)$, $\PSU_3(q)$, or ${^2}G_2(3^{2a+1})'$.
If $G=\PSL_2(q)$, then one computes directly (again
noting that if $q =5$, then $C$ contains all unipotent elements).

If $G=\PSU_3(q)$, then $C$ either consists of transvections or
a suitable pair of elements in $C$ have commutator which is a transvection.
Thus, $C \cap \SL_2(q)$ is nontrivial, a contradiction.
Suppose that $G={^2}G_2(3^{2a+1})$ with $a\ge1$. Note that any unipotent
element is conjugate to an element of ${^2}G_2(3) \cong \PSL_2(8).3$.
In particular, any unipotent element is contained
in at least two maximal subgroups. For ${^2}G_2(3)'=\PSL_2(8)$ note
that any element of order $3$ is contained
in a Frobenius group of order $21$, a contradiction. So any nontrivial element
of $C$ has order $9$. A straightforward computation (see Lemma~\ref{lem:rank1}
below) shows that $C$ cannot be closed under commutators.

Now assume that $G$ has rank at least $2$. Then a Sylow $p$-subgroup is
contained in at least two maximal parabolic subgroups, a contradiction.
\end{proof}

\begin{lem}   \label{lem:rank1}
 $G$ is not a rank $1$ Lie type group in characteristic $r \ne p$.
\end{lem}

\begin{proof}
Let $B$ be a Borel subgroup of $G$ and $U \le B$ its unipotent radical, a
Sylow $r$-subgroup.

We consider the various cases.

Suppose first that $G=\PSL_2(q)$ with $q$ a power of $r$ and $q \ge 7$
(with $q \ne 9$). Since $p$ is odd, it follows that $p$ divides precisely
one of $q \pm 1$. If $p$ divides $q-1$, then $C \cap B$ is nontrivial and
since $O_p(B)=1$, $G$ cannot be a minimal counterexample.

Suppose that $p|(q+1)$. If $q$ is not a power of $3$, we can argue as for the
case $p=2$ to see that for $x\in C$, $\tr[x,x^g]$ can be arbitrary and in
particular, $[x,x^g]$ is not always a $p$-element for some $g \in G$.

So assume that $q=3^e \ge 27$.
Note that $|C^{\#}| \ge q(q-1)$. Also $C \cap B = \{1\}$. Fix $a\ne b\in G/B$.
Let $C(a,b)=\{ x \in C \mid xa = b\}$. For a fixed $a$, since there are only
$q$ possibilities for $b$, we see that $|C(a,b)| \ge q-1$ and since $G$ is
$2$-transitive, in fact we see that $|C(a,b)|=q-1$ for all $a \ne b$.
Let $c$ be a third (distinct) element in $G/B$ and consider
$C(a,b,c):=\{x \in C \mid xa=b, xb=c \}$. If $x \ne y \in C(a,b,c)$, then we
see that $[y,x]$ fixes $a$. Moreover, $x$ and $y$ do not commute for if they
do, then $x^{-1}y$ is in a nonsplit torus and also
in a conjugate of $B$, whence $x=y$, a contradiction. Thus, we are done
unless $|C(a,b,c)| \le 1$ for all $c$ different from $a, b$. On the other hand,
$C(a,b)$ is the disjoint union of the $C(a,b,c)$ for the $q-1$ different
choices for $c$. Thus, we are done unless $|C(a,b,c)|=1$ for all distinct
triples $(a,b,c) \in (G/B)^3$ in which case $|C^{\#}| = q(q-1)$ and $C^{\#}$
is a single conjugacy class. In particular, this implies that
$\tr[x,y] =\pm\tr(x)$ for any noncommuting $x, y\in C$ (working in $\SL_2(q)$).

For $s\in\FF_q^\times$ let $g = g(s)$ be the diagonal matrix with eigenvalues
$s, s^{-1}$. Thus $\tr[x,x^g]$ must take on the same value for at
least $(q-3)/2$ different values of $s$. Note that $f(s):=\tr[x,x^g]$ is an
$\FF_q$-linear combination of $s^4, s^3, \ldots, s^{-3}, s^{-4}$. Write
$f(s)=\sum_{i=-4}^4 a_is^i$. Thus, $f(s) = t$ is fixed for at least $(q-1)/3$
values of $s$. Multiplying through by $s^4$ gives
$s^4f(s)-ts^4$ has at least $(q-3)/2$ zeroes and is a polynomial in $s$
of degree at most $8$. Thus, since $q > 19$, $s^4f(s)=ts^4$ for all $s$,
whence $f(s)=f(1) = \tr(1)=2$. It follows that $\tr(x)=2$. However, the only
elements in $\SL_2(q)$ with trace $2$ are unipotent,
a contradiction.

Suppose that $G=\PSU_3(q)$ with $q \ge 3$. By Lemma~\ref{lem:4},
a nontrivial element of $C$ be must the product of two pseudo-reflections whence
fixes a $1$-space and so either is contained in $\SL_2(q)$ or a Borel
subgroup, a contradiction.

Next suppose that $G=\tw2B_2(q^2)$ with $q^2=2^{2a+1}$, $a\ge1$. Since
every nontrivial element of $C$ is contained in a unique maximal
subgroup, it follows that $p$ divides $q^2 \pm\sqrt{2}q + 1$. Note
that $|G|= q^4(q^4+1)(q^2-1)$. Let $B$ be a Borel subgroup of $|G|$.
If $C$ contains at least two nontrivial
conjugacy classes, then we argue as for the case $\PSL_2(3^e)$
and see that $|C(a,b, c)| > 1$ for some distinct $a,b,c \in G/B$ and
get a contradiction. If $C$ consists of a single nontrivial class, then
we also argue as for $\PSL_2(3^e)$ (conjugating a fixed $x$
by the $q^2-1$ elements in a torus $T \le B$). We conclude that
$\tr(x) = 0$ for all $x \in C$ (in the $4$-dimensional representation).
Now 5-elements have trace~$-1$, while for $p \ne 5$ it is straightforward to
see that nontrivial $p$-elements do not have trace in $\FF_2$, a contradiction.

Finally suppose that $G={^2}G_2(q^2)$ with $q^2=3^{2a+1}$, $a\ge 1$.
Note that the order of $G$ is $q^6(q^6+1)(q^2-1)$. The maximal tori
of $G$ have order $q^2 \pm 1$ or $q^2 \pm\sqrt{3}q + 1$.
In the first two cases, the elements are contained in $\PSL_2(q^2)$,
whence the result follows by minimality. So we may assume that
$p$ divides $q^2 \pm \sqrt{3}q + 1$. Argue precisely as for
$\tw2B_2(q^2)$ to obtain a contradiction.
\end{proof}

\begin{lem}  \label{lem:classical}
 $G$ is not a classical group in characteristic $r \ne p$.
\end{lem}

\begin{proof}
Let $V$ be the natural module for the quasi-simple classical group with
factor group $G$.

If $G=\PSL_n(q)$, then in fact in Lemma~\ref{lem:4} we may choose an involution
in $\PGL_n(q)$ (because it preserves the conjugacy class of any semisimple
element). So we see that a nontrivial element $x \in C$ can be written as
a product of either
two reflections or two transvections, whence $x$ centralizes a subspace
of codimension $2$. Since $x$ is not contained in a proper parabolic subgroup
$P$ (since $O_p(P)$ is trivial), minimality implies $G=\PSL_2(q)$,
contradicting Lemma~\ref{lem:rank1}.

If $G=\PSU_n(q)$, $n \ge 3$, then we see that there is $x\in C$ with
$\dim [x,V] \le 2$ as well.
It follows that $x \in \SL_2(q)$ or is contained in a parabolic subgroup,
a contradiction.

Suppose that $G=\PSp_{2n}(q)$ with $n\ge2$ (note that $\PSp_4(2)'\cong\fA_6$
was already handled). So some (nontrivial) element $x \in C$ is a product
of two involutions with two nontrivial eigenvalues.
Thus, $\dim [x,V] \le 4$. If $x$ has a non-zero fixed space on
$V$, then $x$ is contained in a parabolic subgroup, a contradiction.
So $n=2$ and $C$ intersects $\SL_2(q)$ or $\PSL_2(q^2)$, a contradiction.

Finally, assume that $G$ is an orthogonal group. We can then assume
that $\dim V \ge 7$ (since the smaller orthogonal groups are isomorphic
to groups we have already handled). On the other hand, the argument above
shows that $\dim [x,V] \le 4$ for some $x \in C$. Then $x$ fixes a singular
vector and so is in a parabolic subgroup, a contradiction.
\end{proof}

\begin{lem}   \label{lem:exceptional}
 $G$ is not an exceptional group of Lie type.
\end{lem}

\begin{proof}
Since we have handled the rank one groups, we assume that $G$ has rank at
least $2$.

Assume that $G$ is defined over the field of $q$ elements. Let $1\ne x\in C$.
Note that $x$ is not contained in a proper parabolic subgroup $M$ (by
induction as $F^*(M)=O_r(M)$ where $r \ne p$ is the prime dividing $q$).
Thus, $x$ is a regular semisimple element.

If $G=G_2(q)$, every $p$-element with $p$ not dividing $q$ is contained
in a maximal torus and every maximal torus is contained in a subgroup
$\SL_3(q)$ or $\SU_3(q)$.

Suppose that $G=\tw3D_4(q)$, $q$ odd. Since nontrivial elements of $C$ are
contained in a unique maximal subgroup by Lemma~\ref{lem:unique}, it follows
that $C$ consist of elements in the cyclic maximal torus of order $q^4-q^2+1$.
Let $C_0$ be a conjugacy class contained in $C$. From the generic character
table of $G$ one computes in \Chevie{} \cite{Chv} that $C_0C_0$ contains the
class $D$ of long root elements in $G$. However, on the $8$-dimensional
natural module, long root elements fix a $6$-dimensional space.
Thus, $DD^{-1}$ contains no regular semisimple elements in $G$.
So choose $x_1, x_2 \in C$ so that $x_1x_2 = d$ is a long root element.
Then $[x_1,x_2]= (x_2x_1)^{-1} x_1x_2 \in D^{-1}D$ is not a regular
semisimple element, hence not in $C$.

So we may assume that $G$ has rank at least $4$. Let $z\in G$ be an involution.
By Lemma~\ref{lem:4}, $z$ inverts some element of $C$ and so in particular a
regular semisimple element of $G$. It follows that two suitable conjugates of
$z$ have centralizer in the underlying algebraic group $X$ of dimension less
than $r=\rank(X)$ (since two conjugates of $z$ generate a subgroup containing
a regular semisimple element).

This implies that $2\dim C_X(z) < \dim X + r$, but by inspection there are
involutions in $X$ (defined over the prime field, and inside any
$\tw2E_6(q)$) with bigger centralizer, see Table~\ref{tab:inv}.
\end{proof}

\begin{table}[htbp]
\caption{Involution centralizers}   \label{tab:inv}
\[\begin{array}{|r|cccc|}
\hline
     X& F_4& E_6& E_7& E_8\\
\hline
 C_X(z)'& B_4& D_5& E_6& E_7+A_1\\
 \dim C_X(z)& 36& 46& 79& 136\\
\hline
\end{array}\]
\end{table}

\begin{lem}  \label{lem:spor}
 $G$ is not a sporadic group.
\end{lem}

\begin{proof}
Let $P$ denote a Sylow $p$-subgroup of $G$. If $P$ has order greater than $p$,
then it follows by \cite{As2} that $P$ is not contained in a unique maximal
subgroup unless $p=3$ and $G=J_3$. Considering the structure of this subgroup
shows that $C$ must contain elements of order $3$. No element of
order $3$ is in a unique maximal subgroup.

So we may assume that $P$ has order $p$. If $G$ contains two classes of
involutions, then since each class inverts $y \in C$, it follows that
$y \in C_G(z)$ for some involution $z$. Indeed, then the product of the
two involutions centralizes $y$ and these involutions generate a dihedral
group of order divisible by $4$ (because not all its involutions are conjugate)
and so the central involution in this dihedral group centralizes $y$.
Inspection of the centralizers of involutions (cf. \cite{GL})
shows that $y$ is not in $O_p(C_G(z))$, a contradiction.

Most of the remaining possibilities are listed in Table~\ref{tab}, which for
the relevant primes either gives an overgroup $H>P$ for which the statement
is known by induction, or the statement that $p$-elements are not inverted by
involutions, as would have to be the case by Lemma~\ref{lem:4}  --- here,
$z$ denotes an involution. Note that the Sylow $5$-subgroup of $J_2$ is
elementary abelian of order~25; one of the two classes of cyclic subgroups of
order~5 is contained in $3.\fA_6$, the other in an $\fA_5$.

We are then only left with the following two configurations:

$G=J_1$, $p=19$: here by explicit computation with the 7-dimensional
representation over $\FF_7$ one just exhibits pairs of non-commuting conjugate
elements of order $19$ whose commutator has order prime to $19$.

$G=Ly$; $p=37$ or $67$: one computes directly with the 111-dimensional
representation over $\FF_5$.
\end{proof}

\begin{table}[htbp]
\caption{Sporadic groups}   \label{tab}
\[\begin{array}{|r|r|r|}
\hline
 G& \text{overgroup of $P$}& \text{not inverted}\cr
  &                 & \text{by involution}\cr
\hline
 M_{11}& \SL_2(3)\,(p=3),\ \PSL_2(11)\,(p=5,11)& \\
 M_{22}& \PSL_2(11)\,(p=5,11),\ \fA_7\,(p=7)& \\
 M_{23}& \PSL_2(11)\,(p=5,11)& p=7,23\\
    J_1& 7.3\,(p=3),\ \PSL_2(11)\,(p=5,11),\ C(z)\,(p=7)& \\
    J_2& 3.\fA_6\,(p=5),\ \fA_5\,(p=5)& \\
    J_3& C(z)\,(p=5),\ \PSL_2(17)\,(p=17)& p=19\\
   McL& \PSL_2(11)\,(p=11)& p=7\\
    Ly& 2.\fA_{11}\,(p=7),\ 2.\fA_{11}\,(p=11),\ 5^3.\PSL_3(5)\,(p=31)& \\
    ON& \fA_6\,(p=5),\ \PSL_3(7)\,(p=7),\ J_1\,(p=11,19)& p=31\\
   F_3& G_2(3)\,(p=13),\ \PSU_3(8)\,(p=19)& p=31\\
\hline
\end{array}\]
\end{table}

\section{Commutators of $5$-elements}

We now consider the remainder of Theorem \ref{thm:mainB}. The proof is quite
similar to the previous result -- a bit trickier because of the weaker inductive
hypothesis. We give a sketch.

\begin{thm}  \label{thm:main2}
 Let $G$ be a finite group and $C$ a normal set of $5$-elements
 that is closed under taking commutators. Then
 $\langle C \rangle O_5(G)/O_5(G)$ is a direct product of copies of $\fA_5$.
\end{thm}

Let $G$ be a minimal counterexample (with $|G| +|C|$ minimal). Clearly,
we have that $O_5(G)=1$ and $G = \langle C \rangle$.

\begin{lem}   \label{lem:simple5}
 $G$ is simple.
\end{lem}

\begin{proof}
Let $N$ be a minimal normal subgroup of $G$.

Suppose that $N$ is central. Then $H:=G/N$ is a direct product of copies of
$\fA_5$ by minimality of $G$. If $|N| \ne 2$, then since the Schur multiplier
of $\fA_5$ has order $2$, it follows that $G = N \times H$ and since $G$ is
generated by $5$-elements, we obtain a contradiction.

So $N$ has order $2$ and $G/N$ is a product of more than one $A_5$. Then by
induction $G/Q$ is a product of $A_5$'s where $Q$ is some component.
If $Q = A_5$, then $G$ is a product of $A_5$'s. Thus every component is an
$\SL_2(5)$. Let $x\in C$ and write $x = (x_1,\ldots,x_t)$, where
$x_i\in Q_i$ (modulo some central element). Then conjugating by
$y = (y_1, 1,\ldots,1)$ we have that $[x,x^y]$ is a 5-element and so
$[x_1,x_1^{y_1}]$ is a 5-element in $Q_1$,  but in $\SL_2(5)$, we can arrange
that the commutator has order 10. So $Z(G)=1$.

If $N$ has order prime to $5$, choose $y \in C$ not commuting with $N$
(this is possible since $C$ generates $G$). By coprime action,
$[y, [y,N]]=[y,N]$ and so $1 \ne [y,y^w] \in N$ for some $w \in [y,N]$.
This contradicts our hypothesis that $C$ is closed under taking commutators.

So $N$ is a direct product $L_1 \times \cdots \times L_t$ where $L_i \cong L$
is a nonabelian simple group of order divisible by $5$. Suppose that $t > 1$.

Let $R:=R_1 \times \cdots \times R_t$ be a Sylow $5$-subgroup of $N$. Let
$R \le Q$ be a Sylow $5$-subgroup of $G$. We can choose $y \in Q$ such
that $y$ does not normalize $L_1$.

By \cite[Thm.~X.8.13]{HB}, $J:=N_N(R)/R = J_1 \times \cdots \times J_t$ is
nontrivial. Now consider the group $\langle J, y \rangle$. Then $J$ has order
prime to $5$ and $y$ does not centralize $J$, whence
as above, there exist $h \in J$ with $[y, y^h]$ a nontrivial element of $J$
and so $[y,y^h]$ is not a $5$-element, a contradiction.

So every minimal normal subgroup is a nonabelian simple group.
If $N_1$ and $N_2$ are distinct minimal normal subgroups, then by
induction $G/N_1 $ and $G/N_2$ are both products of $A_5$'s and since $G$
embeds in $H:=G/N_1\times G/N_2$ (and projects onto each simple factor) $G$
itself is also a product of $A_5$'s.

So $G$ has a unique minimal normal subgroup $N$ that is nonabelian
simple. We claim that $G=N$. If not, since $G/N$ is solvable,
it follows that $D$, the set of commutators of elements in $C$
is proper in $C$. If $D=1$, then $G$ is abelian, a contradiction.
Since $O_5(G)=1$, it follows that $\langle D \rangle \cong \fA_5 = N$
and the result holds. So $G=N$ is simple.
\end{proof}

We can assume that every element of $C$ is a commutator of a pair
of elements of $C$ (otherwise replace $C$ by this smaller set
of commutators).

We now can argue in a similar fashion to the proof in the previous section.
One has to do slightly more work (because we cannot appeal to Wielandt's
result).

\begin{lem}
 $G \ne \fA_n, n \ge 5$.
\end{lem}

\begin{proof}
Let $1 \ne x$ be a nontrivial $5$-element.
Let $t$ be an involution moving $4$ points all contained
in a single orbit of $x$ so that $t$ does not invert $x$. Then $[x,x^t] \ne 1$
and as above, this implies that $t$ inverts a nontrivial element of $C$, whence
$x$ must be a $5$-cycle. If $n =5$, the conclusion is allowed and if
$n > 5$, it suffices to check $\fA_6$.
\end{proof}

\begin{lem}
 $G$ is not a finite group of Lie type in characteristic~$5$.
\end{lem}

\begin{proof}
If $G$ has rank $1$, we argue as earlier. If $G$ has rank at least $2$, we
can find a maximal end node parabolic $M$ such that $M \cap C$ is not
contained in $O_5(M)$ \cite[\S2]{GS}, whence the result follows by induction
unless the derived subgroup of the Levi subgroup is $\PSL_2(5)$. This implies
that $G$ is of rank $2$ defined over $\FF_5$ and an easy inspection completes
the proof.
\end{proof}

\begin{lem}
 $G$ is not a finite group of Lie type in characteristic $r \ne 5$.
\end{lem}

\begin{proof}
If some element of $C^{\#}$ centralizes a nontrivial unipotent
element, then $C^{\#}$ intersects a maximal parabolic subgroup $M$ of $G$.
Since $F^*(M)=O_r(M)$, it follows that some $1 \ne y \in C$ normalizes
and does not centralize $O_r(M)$, whence $1 \ne [y,y^x]$ is an $r$-element
for some $x \in O_r(M)$.

Thus, every element of $C^{\#}$ is a regular semisimple element.
If $G$ is classical, it is straightforward to see that we can choose an
involution $y$ which has fixed space of codimension at most $2$ (on the
natural module) such that $[x,x^y] \ne 1$ for some $x \in C$ (just choose $y$
not in the normalizer of the torus that is the centralizer of $x \in C$).
We argue as in the previous section to see that the fixed space
of some nontrivial element of $C$ is large, whence the rank is quite small.
The analysis of the small rank cases gives the only example.

If $G$ is exceptional, the proof is essentially as in the general case as well.
Namely, if $G$ has rank at least $4$, then choose an involution $z$
that does not invert any regular semisimple element. However, any involution
in $G$ does invert a nontrivial element of $C$.

For the rank one and two groups, we argue precisely as in the case of $p\ne 5$.
\end{proof}

\begin{lem}
 $G$ is not a sporadic group.
\end{lem}

\begin{proof}
By inspection of subgroups, we see that $C$ must contain
a class of nontrivial $5$-central elements (this class is often unique).
One can produce an overgroup of such an element where the result holds
by induction (with $O_5$ trivial and not containing a normal
product of $\fA_5$ subgroups).
\end{proof}


\end{document}